\documentclass{amsart}
\usepackage{amssymb}
\usepackage{chemarr}
\usepackage{bbm}
\usepackage[latin1]{inputenc}
\usepackage[english]{babel}
\usepackage{enumerate}
\usepackage{ifpdf}
\ifpdf
  \usepackage[pdftex]{graphicx}
  \usepackage{epstopdf}
\else
  \usepackage[dvips]{graphicx}
\fi
\usepackage[pdftitle={Multiscale splitting for spatial stochastic kinetics},
  pdfauthor={Augustin Chevallier and Stefan Engblom},
  pdffitwindow=true,
  breaklinks=true,
  colorlinks=true,
  urlcolor=blue,
  linkcolor=red,
  citecolor=red,
  anchorcolor=red]{hyperref}
\usepackage[numbers,sort]{natbib}

\numberwithin{equation}{section}
\numberwithin{table}{section}
\numberwithin{figure}{section}

\theoremstyle{plain}
\newtheorem{theorem}{Theorem}[section]
\newtheorem{lemma}[theorem]{Lemma}

\theoremstyle{definition}
\newtheorem{definition}{Definition}[section]
\newtheorem{assumption}[definition]{Assumption}
\newtheorem{condition}[definition]{Condition}

\theoremstyle{remark}
\newtheorem*{remark}{Remark}

\newcommand{\stoich}{\mathbb{N}}
\newcommand{\Realdom}{\mathbf{R}}
\newcommand{\Intdom}{\mathbf{Z}}

\newcommand{\cadlag}{c{\`a}dl{\`a}g}
\DeclareMathOperator{\Expect}{\mathbb{E}}
\newcommand{\Prob}{\mathbf{P}}
\newcommand{\Probspace}{\Omega}
\newcommand{\Probelem}{\omega}
\newcommand{\Probfiltr}{\mathcal{F}}
\newcommand{\stopping}{P}
\newcommand{\Sstopping}{\bar{P}}
\newcommand{\StopT}[1]{\hat{#1}}
\newcommand{\loc}{\mathrm{loc}}
\newcommand{\Sspace}[1]{S_{\Probfiltr}^{#1,\loc}(\Intdom_{+}^{D \times J})}
\newcommand{\SspaceS}[1]
           {S_{\Probfiltr}^{#1,\loc}(S^{-1}\Intdom_{+}^{D \times J})}
\newcommand{\SspaceZ}[1]
           {S_{\Probfiltr}^{#1,\loc}([\Intdom_{+}^{D_1 \times J};\,
               \Realdom_{+}^{D_2 \times J}])}

\newcommand{\Vol}{V_{tot}}
\newcommand{\VoxVol}{V}
\newcommand{\MeshSet}{\mathcal{M}}
\DeclareMathOperator{\diam}{diam}
\DeclareMathOperator{\diag}{diag}

\newcommand{\ScaleS}{S}

\newcommand{\ScaleX}{\bar{X}}
\newcommand{\ScaleY}{{\bar{Y}^{(h)}}}
\newcommand{\ScaleZ}{\bar{Z}}
\newcommand{\Group}{G}

\newcommand{\wvec}{\boldsymbol{w}}
\newcommand{\wvect}{\boldsymbol{w}^{T}}

\newcommand{\lvec}{\boldsymbol{l}}
\newcommand{\lvect}{\boldsymbol{l}^{T}}

\newcommand{\onevec}{\boldsymbol{1}}

\newcommand{\one}{\mathbbm{1}}
\newcommand{\onet}{\mathbbm{1}^T}

\newcommand{\wnorm}[1]{\left\| #1 \right\|_{\wvec}}

\newcommand{\wunorm}[1]{\left\| #1 \right\|_{\wvec,1}}

\newcommand{\lnorm}[1]{\left\| #1 \right\|_{\lvec}}

\newcommand{\lunorm}[1]{\left\| #1 \right\|_{\lvec,1}}
\newcommand{\pnorm}[1]{\left\| #1 \right\|_{2}}
\newcommand{\LipPr}{\bar{L}}
\newcommand{\lec}{\le_{C}}

\newcommand{\lipa}{a}
\newcommand{\lipb}{b}

\begin{document}

\title[Multiscale splitting for spatial stochastic kinetics]{Pathwise
  error bounds in Multiscale variable splitting methods for spatial
  stochastic kinetics}

\author[A. Chevallier]{Augustin Chevallier}
\address{INRIA Sophia Antipolis \\
  2004 Route des Lucioles,\\
  06902 Valbonne, France.}
\email{augustin.chevallier@ens-cachan.org}

\author[S. Engblom]{Stefan Engblom}
\address{Division of Scientific Computing \\
  Department of Information Technology \\
  Uppsala University \\
  SE-751 05 Uppsala, Sweden.}
\urladdr{\url{http://user.it.uu.se/~stefane}}
\email{stefane@it.uu.se}
\thanks{Corresponding author: S. Engblom, telephone +46-18-471 27 54,
  fax +46-18-51 19 25.}

\keywords{Hybrid mesoscopic model; Mean square bounds; Continuous-time
  Markov chain; Jump process; Rate equation}

\subjclass[2010]{Primary: 65C20, 60J22; Secondary: 65C40, 60J27}
%
%

\date{\today}

\begin{abstract}
  Stochastic computational models in the form of pure jump processes
  occur frequently in the description of chemical reactive processes,
  of ion channel dynamics, and of the spread of infections in
  populations. For spatially extended models, the computational
  complexity can be rather high such that approximate multiscale
  models are attractive alternatives. Within this framework some
  variables are described stochastically, while others are
  approximated with a macroscopic point value.

  We devise theoretical tools for analyzing the pathwise multiscale
  convergence of this type of variable splitting methods, aiming
  specifically at spatially extended models. Notably, the conditions
  we develop guarantee well-posedness of the approximations without
  requiring explicit assumptions of \textit{a priori} bounded
  solutions. We are also able to quantify the effect of the different
  sources of errors, namely the \emph{multiscale error} and the
  \emph{splitting error}, respectively, by developing suitable error
  bounds. Computational experiments on selected problems serve to
  illustrate our findings.
\end{abstract}

\selectlanguage{english}

\maketitle


\section{Introduction}

Mesoscopic spatially extended stochastic models are in frequent use in
many fields, with notable examples found in cell biology,
neuroscience, and epidemiology. The traditional macroscopic
description is a partial differential equation (PDE) governing the
flow of concentration field variables in a generalized
reaction-transport process. Whenever a certain concentration is small
enough, discrete stochastic effects become more pronounced, thus
invalidating the assumptions behind the macroscopic model. An
alternative is then to turn to a mesoscopic stochastic model, a
continuous-time Markov chain over a discrete state-space. This model
often remains accurate at an acceptable computational complexity.

In the traditional non-spatial, or well-stirred setting, early work by
Kurtz connected theses two descriptions via limit theorems, showing
essentially that continuous approximations emerge in the limit of
large molecular numbers, sometimes referred to as the ``thermodynamic
limit''. Strong approximation theorems in the same setting were later
also developed (for more of this, see the monograph \cite{Markovappr}
and the references therein).

\emph{Multiscale-}, or hybrid descriptions, in which the two scales
are blended has since attracted many researchers. The focus of the
research tend to fall into one of two categories; either
``theoretical'' and concerning error bounds and rate of convergence,
or more ``practical'' by developing actual implementations and general
software.

In the first category, tentative analysis of specific examples are
found in \cite{kurtz_multiscale}, while \cite{plyasunov_mscaleMMS,
  TJahnkeMMS2012} are of more general character and based on averaging
techniques, and conditional expectations, respectively. A related
analysis in the sense of mean-square convergence for operator
splitting techniques is found in \cite{jsdesplit}. In
\cite{kurtz_multiscale2} the issue of a proper scaling is stressed and
similar remarks are made in \cite{AGangulyMMS2014}, where notably, a
practical multiscale simulation algorithm is also devised.

Towards the more algorithmic side, an early suggestion for a hybrid
method in \cite{haseltine_HSSA} came to be followed up by several
others \cite{adaptiveHSSA, Hy3S, HyMSMC}. Related multiscale
algorithms based on quasi equilibrium assumptions are found in
\cite{nestedSSA, slowSSA}, and the method in \cite{ssa_parareal}
relied on the macroscale description as a preconditioner to bring out
parallelism.

With few exceptions \cite{parallel_LKMC, pseudo_compartment}, the main
body of work has been done in the well-stirred (or 0-dimensional)
setting. Since the work \cite{master_spatial} and the software
described in \cite{URDMEpaper}, however, it is fairly well understood
how \emph{spatial} models are to be developed. Here the computational
complexity is much higher such that multiscale methods appear as a
very attractive alternative. This is the starting point for the
present contribution.

The goal with the analysis of the paper is twofold. We will firstly
deal with the multiscale analysis required for the splitting of the
state variable into a stochastic and a deterministic part,
respectively. Secondly, we will also deal with the numerical analysis
relied upon when designing a basic but representative
time-discretization of this approximating process.

The paper is organized as follows: below we first summarize the main
results of the paper. In \S\ref{sec:jsdes} we work through the
description of mesoscopic reactive processes as continuous-time Markov
chains with a focus on the spatial case. A substantial effort is made
to avoid any possibly circular assumptions on the solution regularity,
but rather to prove all results within a single coherent
framework. The analysis of the multiscale approximation is found in
\S\ref{sec:analysis}, where error bounds for both the
\emph{multiscale} and the \emph{splitting} errors are developed. Our
approach is pathwise in the sense that the errors are measured in
$L_2$ over a single probability space. Selected numerical examples are
presented in \S\ref{sec:examples}, and a concluding discussion is
offered in \S\ref{sec:conclusions}.

\subsection{Summary of main results}

A brief orientation of the technical results of the paper is as
follows:
\begin{enumerate}
\item \label{it:stab1} Theorem~\ref{th:exist0} proves a strong
  regularity result for the type of spatial reactive processes
  considered in the paper.
\item Theorem~\ref{th:exist1} proves the corresponding result in the
  setting of a multiscale framework. In particular, this reveals
  partial assumptions for when a multiscale description is meaningful.
\item \label{it:stab2} Theorems~\ref{th:exist2} and \ref{th:exist3}
  similarly develop regularity results for the \emph{multiscale-} and
  the \emph{split-step approximations}, respectively.
\item \label{it:conv1} Theorems~\ref{th:ScalingErrorBounded} and
  \ref{th:ScalingError} provide for a multiscale convergence theory
  when parts of the dynamics is approximated via deterministic terms.
\item \label{it:conv2} Theorems~\ref{th:spliterr2} and
  \ref{th:spliterr1} similarly provides for a convergence theory of
  split-step methods in a general multiscale setting.
\end{enumerate}
In this list, items \ref{it:stab1}--\ref{it:stab2} proves
well-posedness and stability for the various involved
processes. Following the celebrated Lax principle, items
\ref{it:conv1}--\ref{it:conv2} next proves convergence and error
estimates by an investigation of the consistency in the different
approximations.


\section{Mesoscopic spatial stochastic kinetics}
\label{sec:jsdes}

We devote this section to some technical developments;
\S\ref{subsec:rdme}--\ref{subsec:mesh} summarize reaction-transport
type modeling over irregular lattices, and regularity results under
suitable model assumptions are developed in
\S\ref{subsec:RDMEstab}. The variable splitting setup to be studied is
similarly detailed in \S\ref{subsec:scaling}--\ref{subsec:splitting},
where the corresponding regularity results are evaluated anew.

Throughout the paper we shall remain in the framework of
continuous-time Markov processes on a discrete state-space, albeit
with some special structure imposed from the spatial context. Assuming
a process $X(t) \in \Intdom_{+}^{D}$ counting at time $t$ the number
of entities in each of $D$ compartments, a set of $R$ state
transitions $X \mapsto X-\stoich_{r}$ is generally prescribed by
\begin{align}
  \label{eq:prop}
  \Prob\left[X(t+dt) = x-\stoich_{r}| \; X(t) = x\right] &=
  w_{r}(x) \, dt+o(dt),
\end{align}
for $r = 1\ldots R$. To enforce a conservative chain which remains in
$\Intdom_{+}^{D}$, we assume $w_{r}(x) = 0$ whenever
$x-\stoich_{r} \not \in \Intdom_{+}^{D}$.

\subsection{Continuous-time Markov chains on irregular lattices}
\label{subsec:rdme}

In the traditional well-stirred setting we have $D$ species
interacting according to $R$ chemical reactions in some fixed volume
$\Vol$. Given an initial state $X(0)$, the dynamics is then fully
described by the \emph{stoichiometric matrix} $\stoich \in \Intdom^{D
  \times R}$, and $w(x) \equiv [w_{1}(x), \ldots, w_{R}(x)]^{T}$, the
set of \emph{propensities}. Assuming a probability space
$(\Probspace,\Probfiltr,\Prob)$ supporting $R$-dimensional Poisson
processes, the state is evolved according to
\cite[Chap.~6.2]{Markovappr}
\begin{align}
  \label{eq:Poissrepr}
  X_i(t) &= X_i(0)-\sum_{r = 1}^{R} \stoich_{ri} \Pi_{r}
  \left(  \int_{0}^{t} w_{r}(X(s)) \, ds \right),
\end{align}
for species $i = 1\ldots D$ and with standard unit-rate independent
Poisson processes $\Pi_{r}$, $r = 1\ldots R$.

If the assumption of a spatially uniform distribution no longer holds
a notation for spatial dependency needs to enter. The given continuous
volume $\Vol$ is discretized into $J$ smaller voxels $(\VoxVol_j)_{j =
  1}^J$ and the state $X \in \Intdom_{+}^{D \times J}$, where $X_{ij}$
is the number of molecules of the $i$th species in the $j$th
voxel. The assumption of global homogeneity is replaced with a
\emph{local} assumption about uniformity in each voxel such that the
dynamics \eqref{eq:Poissrepr} may be used anew on a per-voxel
basis. Adding suitable terms covering any specified transport process
we get
\begin{align}
  \label{eq:RDMEPoissrepr}
  X_{ij}(t) = X_{ij}(0) &-
  \sum_{r = 1}^{R} \stoich_{ri} \Pi_{rj}  
  \left( \int_{0}^{t} w_{rj}(X_{\cdot ,j}(s)) \, ds \right) \\
  \nonumber
  &-\sum_{k = 1}^{J} \Pi_{ijk}' 
  \left( \int_{0}^{t} q_{ijk}X_{ij}(s) \, ds \right) \\
  \nonumber
  &+\sum_{k = 1}^{J} \Pi_{ikj}'
  \left( \int_{0}^{t} q_{ikj}X_{ik}(s) \, ds \right),
\end{align}
where $q_{ijk}$ is the rate per unit of time for species $i$ to move
from the $j$th voxel to the $k$th.

An important consequence of the integral representation
\eqref{eq:RDMEPoissrepr} is \emph{Dynkin's formula}
\cite[Chap.~9.2.2]{BremaudMC}. For $f:\Intdom_{+}^{D \times J} \to
\Realdom$ a suitable function,
\begin{align}
  \nonumber
  \Expect&\left[f(X(\StopT{t}))-f(X(0))\right] = \\
  \label{eq:Dynkin}
  &\Expect\Biggl[ \int_{0}^{\StopT{t}}
    \sum_{j = 1}^J \sum_{r = 1}^{R} w_{rj}(X_{\cdot,j}(s)) \left[ 
      f(X(s)-\stoich_{r}\onet_j)-
      f(X(s)) \right] \, ds \Biggr] \\
  \nonumber
  +&\Expect\Biggl[ \int_{0}^{\StopT{t}} \sum_{j,k = 1}^J \sum_{i = 1}^D
    q_{ijk} X_{ij}(s) \left[ f(X(s)-\one_i\onet_j+\one_i\onet_k)-
      f(X(s)) \right] \, ds \Biggr] \\
  \nonumber
  +&\Expect\Biggl[ \int_{0}^{\StopT{t}} \sum_{j,k = 1}^J \sum_{i = 1}^D
    q_{ikj} X_{ik}(s) \left[ f(X(s)+\one_i\onet_j-\one_i\onet_k)-
      f(X(s)) \right]\, ds \Biggr],
\end{align}
expressed in terms of the stopped process $X(\StopT{t}) = X(t \wedge
\tau_{\stopping})$ for a stopping time $\tau_{\stopping} := \inf_{t
  \ge 0}\{\|X(t)\| > \stopping\}$ in some suitable norm, and $P > 0$
an arbitrary real number. In \eqref{eq:Dynkin}, $\one_j$ is an
all-zero column vector of suitable height and with a single 1 at
position $j$.

\subsection{Mesh regularity}
\label{subsec:mesh}

The subdivision of the total volume $\Vol$ into smaller voxels is in
principle arbitrary. However, any meaningful analysis will clearly
depend to some extent on the regularity of this discretization.

\begin{definition}[\textit{Mesh regularity parameters}]
  \label{def:mesh}
  We consider a geometry in $d$ dimensions and total volume $\Vol$,
  discretized by any member in the set of meshes $\MeshSet$. For any
  such mesh $M \in \MeshSet$ consisting of voxel volumes
  $(\VoxVol_j)_{j=1}^{J}$ we assume that it holds that
  \begin{align}
    \label{eq:uniform_mesh}
    &m_\VoxVol  \bar{\VoxVol}_{M} \le \VoxVol_j \le M_\VoxVol \bar{\VoxVol}_{M}, \\
    \label{eq:voxel_shape}
    &m_h \VoxVol_j^{1/d} \le \diam(\VoxVol_j) \le M_h
    \VoxVol_j^{1/d}, \\
    \label{eq:connectivity}
    &|\{k; \; q_{ijk} \not = 0\}| \leq M_D,
    \end{align}
    for constants $0 < m_\VoxVol \le M_\VoxVol$, $0 < m_h \le M_h$,
    $M_D$, and average voxel volume $\bar{\VoxVol}_{M} = J^{-1}\sum_{j
      = 1}^{J} V_{j}$. Hence under this parametrization we may write
    \begin{align*}
      \MeshSet = \MeshSet(m_\VoxVol,M_\VoxVol,m_h,M_h,M_D).
    \end{align*}
\end{definition}

Informally, \eqref{eq:uniform_mesh} measures how far the meshes in
$\MeshSet$ are from being uniform, \eqref{eq:voxel_shape} ensures that
no single voxel collapses into a voxel in less than $d$ dimensions,
and \eqref{eq:connectivity} that the connectivity of the mesh is
bounded. In the present paper \eqref{eq:voxel_shape} is not used
explicitly; this assumption assures a connection to the macroscopic
viewpoint in that a concentration variable may be meaningfully defined
everywhere.

\subsection{Solution regularity}
\label{subsec:RDMEstab}

We next ensure the well-posedness of \eqref{eq:RDMEPoissrepr} by
deriving some pathwise bounds on this process. To get some feeling for
what is going on we first look briefly at the corresponding
PDE-setting.

Assume for simplicity that the transport rates $q_{ijk}$ have been
chosen as a consistent discretization of the operator
$\sigma_i \Delta$ under homogeneous Neumann conditions at the mesh
$M$. Denoting a deterministic time-dependent concentration variable by
$v_{i} = v_i(t,x)$ for $x \in \Realdom^d$ and $i = 1\ldots D$, a
macroscopic reaction-diffusion PDE corresponding to
\eqref{eq:RDMEPoissrepr} reads
\begin{align}
  \label{eq:RDPDE}
  \frac{\partial v_i}{\partial t} &= \sigma_{i}\Delta v_i-
  \sum_{r = 1}^{R} \stoich_{ri} u_{r}(v_{\cdot}), 
  \mbox{ in $\Vol$, } 
  \quad \frac{\partial v_i}{\partial n} = 0, \mbox{ on $\partial \Vol$},
\end{align}
for certain nonlinear rates $u_r$, $r = 1...R$ to be prescribed below.
Equipped with suitable initial data, \eqref{eq:RDPDE} can be expected
to be a well-posed initial-boundary value problem in
$L^{\infty}([0,T]) \times L^{p}(\Vol)$ for any $p \ge 1$.

For the stochastic case \eqref{eq:RDMEPoissrepr}, and in the
non-spatial setting, an analysis in the form of assumptions and
various \textit{a priori} bounds has been developed previously
\cite{jsdestab}. We borrow many ideas from this work in what follows.

The propensities in \eqref{eq:RDMEPoissrepr} generally obey the
\emph{density dependent} scaling such that
$w_{rj}(x) = \VoxVol_{j} u_{r}(\VoxVol_{j}^{-1}x)$ for some
dimensionless function $u_{r}$ \cite[Chap.~11]{Markovappr}. We further
expect from a physically realistic model that the number of molecules
in an isolated volume $\VoxVol_j$ can somehow be bounded \textit{a
  priori}. To this end we postulate the existence of a weighted norm
\begin{align}
  \label{eq:ldef}
  \wnorm{x} := \wvect x, \quad x \in \Realdom_{+}^{D},
\end{align}
normalized such that $\min_{i} \wvec_{i} = 1$. Following
\cite{jsdestab} we formulate

\begin{assumption}[\textit{Reaction regularity}]
  \label{ass:rbound}
  For a mesh $M \in \MeshSet$ consisting of voxel volumes
  $(\VoxVol_j)_{j=1}^{J}$ we assume the density dependent scaling,
  \begin{align}
    \label{eq:ddepend}
    w_{rj}(x) &= \VoxVol_{j}
    u_{r}(\VoxVol_{j}^{-1}x), \\
    \intertext{where $u$ is \emph{independent of the mesh} and further
      satisfies,}
    \label{eq:1bnd}
    -\wvect \stoich u(x) &\le A+\alpha \wnorm{x}, \\
    \label{eq:bnd}
    (-\wvect \stoich)^{2} u(x)/2 &\le B+\beta_{1}
    \wnorm{x}+\beta_{2}\wnorm{x}^{2}, \\
    \label{eq:lip}
    |u_{r}(x)-u_{r}(y)| &\le
    L_{r}(\stopping)\|x-y\|, \mbox{ for $r = 1\ldots R$, and $\wnorm{x}
    \vee \wnorm{y} \le \stopping$}.
  \end{align}
  With the exception of $\alpha$, all parameters
  $\{A,B,\beta_{1},\beta_{2},L\}$ are assumed to be non-negative.
\end{assumption}

When considering spatially varying solutions, the natural analogue to
\eqref{eq:ldef} is
\begin{align}
  \label{eq:l1norm}
  \wunorm{X} &\equiv \sum_{j = 1}^{J} \wnorm{X_{\cdot,j}} = \wvect X \onevec,
\end{align}
for $\onevec$ an all-unit column vector of suitable height. Our
starting point is Dynkin's formula \eqref{eq:Dynkin}. We find
\begin{align}
  \nonumber
  \Expect\left[\wunorm{X(\StopT{t})}^{p}\right] &= 
  \Expect\left[\wunorm{X(0)}^{p}\right]+
  \Expect\left[ \int_{0}^{\StopT{t}} F(X(s)) \, ds \right] \\
  \intertext{where}
  \label{eq:Fdef}
  F(X) &\equiv \sum_{j = 1}^J \sum_{r = 1}^{R} w_{rj}(X_{\cdot,j}(s)) \left[ 
    \left( \wunorm{X(s)}-\wvect \stoich_{r} \right)^{p}-
    \wunorm{X(s)}^{p} \right].
\end{align}
We quote the following convenient inequality.

\begin{lemma}[\textit{Lemma~4.6 in \cite{jsdestab}}]
  \label{lem:diffbound}
  Let $f(x) \equiv (x+y)^{p}-x^{p}$ with $x \in \Realdom_{+}$ and $y
  \in \Realdom$. Then for integer $p \ge 1$ we have the bounds
  \begin{align}
    \label{eq:signed_diffbound}
    f(x) &\le p y x^{p-1}+
    2^{p-4} p(p-1) y^{2} \left[ x^{p-2}+
      |y|^{p-2} \right], \\
   \label{eq:abs_diffbound}
    |f(x)| &\le p |y| 2^{p-2} \left[ x^{p-1}+
      |y|^{p-1} \right].
  \end{align}
\end{lemma}

Using Lemma~\ref{lem:diffbound} \eqref{eq:signed_diffbound},
Assumption~\ref{ass:rbound}~\eqref{eq:ddepend}--\eqref{eq:bnd}, and
Definition~\ref{def:mesh}~\eqref{eq:uniform_mesh} we obtain, where for
brevity $x \equiv \wunorm{X}$,
\begin{align}
  \label{eq:Fbound}
  F(X) &\le p(A\Vol+\alpha x) x^{p-1}+
  C_p(B\Vol+\beta_{1} x+
  \beta_{2}^\VoxVol x^{2})
  (x^{p-2}+C_\stoich^{p-2}),
\end{align}
where $C_p := 2^{p-3} p(p-1)$, $\beta_{2}^\VoxVol := \beta_{2}
m_\VoxVol^{-1}\bar{\VoxVol}_M^{-1}$, and $C_\stoich := \|\wvect
\stoich\|_{\infty}$. Combining \eqref{eq:Fdef} and \eqref{eq:Fbound}
and using Young's inequality several times we may obtain a bound of
the form
\begin{align}
  \label{eq:momest}
  \Expect\left[\wunorm{X(\StopT{t})}^{p}\right] &\le 
  \Expect\left[\wunorm{X(0)}^{p}\right]+\Expect\Biggl[ \int_{0}^{t}
    C(1+\wunorm{X(\StopT{s})}^{p}) \, ds \Biggr],
\end{align}
\noindent
for some $C > 0$. Using Gronwall's inequality and letting $\stopping
\to \infty$ we arrive at
\begin{theorem}
  \label{th:Epbound}
  Let $X(t)$ obey \eqref{eq:RDMEPoissrepr} under
  Assumption~\ref{ass:rbound}. Then for any integer $p \ge 1$,
  \begin{align}
    \Expect\left[\wunorm{X(t)}^{p}\right] &\le
    \left( \Expect\left[\wunorm{X(0)}^{p}\right] +1\right)\exp(Ct)-1,
  \end{align}
  where the constant $C > 0$ depends on $p$ and on the constants in
  the assumptions.
\end{theorem}

\begin{proof}
  It remains to prove that $\StopT{t} \to t$ almost surely as
  $\stopping \to \infty$. Suppose to the contrary that
  $\StopT{t} = \tau_\stopping \wedge t$ does not converge a.s.~to $t$
  as $\stopping \to \infty$. Define
  $A \equiv \{\Probelem; \; \forall P: \tau_\stopping(\Probelem) <
  t\}$.
  By the assumption $\Prob(A) > 0$ and, for any $\Probelem \in A$, and
  for all $\stopping > 0$,
  \begin{align*}
    \sup_{0 \leq s \leq t} \|X_s(\Probelem)\| >
    \stopping \mbox{, or simply, }
    \sup_{0 \leq s \leq t} \|X_s\|(\Probelem) = \infty.
  \end{align*}
  In other words, $X(\StopT{t},\omega) \to \infty$ for every $\omega
  \in A$, and $\|X(\StopT{t},\omega)\|$ forms an increasing sequence
  with respect to $\stopping$. Using the Lebesgue monotone
  convergence theorem together with $\Prob(A) > 0$, we get that
  $E[\|X(\StopT{t})\|] \geq E[\|X(\StopT{t})\| 1_{\omega \in A}] \to
  \infty$. However, $E[\|X(\StopT{t})\|]$ is bounded from above
  independently of $\stopping$ and thus we have a contradiction.
\end{proof}

Notably, when small voxels $\VoxVol_j$ are present and quadratic
reactions which are not $\wvec$-neutral are allowed (i.e.~$\beta_2
\not = 0$), then an investigation of $C$ in \eqref{eq:momest} reveals
that the second order moment and higher may grow fast as $\exp(\beta_2
\VoxVol_j^{-1}t)$.

To achieve pathwise convergence results we will need a stronger
regularity guarantee which requires control of the martingale part via
Burkholder's inequality. To this end we define the quadratic variation
of a real-valued process $(Y_{t})_{t \ge 0}$ by
\begin{align}
  \label{eq:QVdef}
  [Y]_{t} &= \lim_{\|\mathcal{P}\| \to 0} \sum_{k = 0}^{n-1}
  \left( Y_{t_{k+1}}-Y_{t_{k}} \right)^{2},
\end{align}
where the partition $\mathcal{P} = \{0 = t_{0} < t_{1} < \cdots <
t_{n} = t\}$ for which $\|\mathcal{P}\| := \max_{k} |t_{k+1}-t_{k}|$
and where the limit is in probability.

\begin{lemma}
  \label{lem:quadratic_variation}
  Let $X(t)$ satisfy \eqref{eq:RDMEPoissrepr} under
  Assumption~\ref{ass:rbound}. Then the quadratic variation of
  $\wunorm{X(t)}^{p}$ is bounded by
  \begin{align}
    \label{eq:quadratic_variation}
    \Expect\left([\wunorm{X}^{p}]_{t}^{1/2}\right) &\le \Expect\left[ \int_{0}^{t}
    C(1+\wunorm{X(s)}^{p}+ \beta_{2}^{\VoxVol}\wunorm{X(s)}^{p+1}) \, ds \right],
  \end{align}
  where $C > 0$ again depends on $p$ and on the constants in
  Assumption~\ref{ass:rbound}, but not on the mesh resolution, and
  where
  $\beta_{2}^{\VoxVol} := \beta_2 m_\VoxVol^{-1}\bar{\VoxVol}_M^{-1}$.
\end{lemma}

\begin{proof}
  Let $t_0 = 0$ and $t_i$ for $i = 1, 2, ...$ be the successive jump
  times of $X$. Then
  \begin{align*}
    [\wunorm{X}^{p}]_{\StopT{t}}^{1/2} &= \left( \sum_{0 < t_i \leq \StopT{t}} 
    \left( \wunorm{X (t_i)}^{p} - \wunorm{X(t_{i-1})}^{p} \right)^2 \right)^{1/2}.
  \end{align*}
  Under the stopping time $X$ is non-explosive with probability 1 and
  the number of jumps is finite in $[0,\StopT{t}]$. Thus we can use
  the inequality $\|\cdot\|_2 \leq \|\cdot\|_1$ to get
  \begin{align*}
    [\wunorm{X}^{p}]_{\StopT{t}}^{1/2} &\leq \sum_{0 < t_i \leq \StopT{t}} \left| \wunorm{X (t_i)}^{p} - \wunorm{X(t_{i-1})}^{p} \right|.
  \end{align*}
  The right-hand side can be written as a Lebesgue-Steiltjes integral,
  \begin{align*}
    \int_{0}^{\StopT{t}}
    \sum_{j = 1}^J \sum_{r = 1}^{R} \left| 
    \left( \wunorm{X(s)}-\wvect \stoich_{r} \right)^{p}-
    \wunorm{X(s)}^{p} \right| \, dY_{rj}(s),
  \end{align*}
  with $Y_{rj}$ the counting process
  $Y_{rj}(t) = \Pi_{rj}\left( \int_0^t w_{rj}(X_{.,j}(s)) \, ds
  \right)$. Taking the expectation yields
  {\small \begin{align*}
    \Expect \left( [\wunorm{X}^{p}]_{\StopT{t}}^{1/2} \right) &\leq 
    \Expect \left[ \int_{0}^{\StopT{t}}
    \sum_{j = 1}^J \sum_{r = 1}^{R} w_{rj}(X_{.,j}(s)) \left| 
    \left( \wunorm{X(s)}-\wvect \stoich_{r} \right)^{p}-
    \wunorm{X(s)}^{p}
    \right|\, ds\right].
  \end{align*}}
  Using Lemma~\ref{lem:diffbound}~\eqref{eq:abs_diffbound} and
  Assumption~\ref{ass:rbound}~\eqref{eq:ddepend} and \eqref{eq:bnd},
  \begin{align*}
    &\le \Expect \left[ \int_{0}^{\StopT{t}} 
      \sum_{j,r} p |\wvect \stoich_{r}| w_{rj}(X_{\cdot,j}(s)) \; 2^{p-2}
      \left[ \wunorm{X(s)}^{p-1}+|\wvect \stoich_{r}|^{p-1} \right]
      \, ds \right] \\
    &\le \Expect \left[ \int_{0}^{\StopT{t}} 
      C_p(B\Vol+\beta_{1}\wunorm{X(s)}+\beta_{2}^\VoxVol\wunorm{X(s)}^{2})
      (\wunorm{X(s)}^{p-1}+C_\stoich^{p-1}) \, ds \right].
  \end{align*}
  Relying on the moment bound in Theorem~\ref{th:Epbound} we let
  $\stopping \to \infty$ to arrive at the stated bound.
\end{proof}

We consider the following strong sense of pathwise locally bounded
processes:
\begin{align}
  \Sspace{p} &= \left\{ X(t,\Probelem): \begin{array}{l}
    X(t) \in \Intdom_{+}^{D \times J} \mbox{ is
      $\Probfiltr_{t}$-adapted such that } \\
    \Expect[\sup_{t \in [0,T]} \wunorm{X_{t}}^{p}] <
    \infty \mbox{ for } \forall T < \infty \end{array} \right\}.
\end{align}

\begin{theorem}[\textit{Regularity}]
  \label{th:exist0}
  Let $X(t)$ be a solution to \eqref{eq:RDMEPoissrepr} under
  Assumption~\ref{ass:rbound} with $\beta_{2} = 0$. Then if
  $\Expect[\wunorm{X(0)}^{p}] < \infty$, $\{X(t)\}_{t \ge 0}
  \in \Sspace{p}$. If $\beta_{2} > 0$ then the conclusion remains
  under the additional requirement that $\Expect[\wunorm{X(0)}^{p+1}]
  < \infty$.
\end{theorem}

\begin{proof}
  This result follows as a combination of Theorem~\ref{th:Epbound} and
  Lemma~\ref{lem:quadratic_variation}. We find that
  \begin{align*}
    \wunorm{X(\StopT{t})}^{p} &= \wunorm{X(0)}^{p}+\int_{0}^{\StopT{t}} 
    F(X(s)) \, ds+M_{\StopT{t}},
  \end{align*}
  with $F$ defined in \eqref{eq:Fdef}. The quadratic variation of the
  local martingale $M_{\StopT{t}}$ can be estimated via
  Lemma~\ref{lem:quadratic_variation},
  \begin{align}
    \label{eq:Mqv}
    \Expect \left( [M]_{\StopT{t}}^{1/2} \right) &\le 
    \Expect \left[ \int_{0}^{\StopT{t}} C(1+\wunorm{X(s)}^{p}+
    \beta_{2}^{\VoxVol}\wunorm{X(s)}^{p+1}) \, ds \right].
  \end{align}
  Assume first that $\beta_{2} = 0$. Using the previously developed
  bound in \eqref{eq:Fbound} and \eqref{eq:momest} for the drift part
  we get
  \begin{align*}
    \wunorm{X(\StopT{t})}^{p} &\le \wunorm{X(0)}^{p}+\int_{0}^{\StopT{t}} 
    C(1+\wunorm{X(s)}^{p}) \, ds+|M_{\StopT{t}}|. \\
    \intertext{Combining with \eqref{eq:Mqv} we find after using 
      Burkholder's inequality \cite[Chap.~IV.4]{protterSDE},}
    \Expect\left[ \sup_{s \in [0,\StopT{t}]} \wunorm{X(s)}^{p}\right] &\le
    \Expect[\wunorm{X(0)}^{p}]+\int_{0}^{\StopT{t}}
    C\left(1+\Expect\left[\sup_{s' \in [0,s]}
      \wunorm{X(s')}^{p}\right] \right)\, ds. \\
    \intertext{For clarity, writing $\wunorm{X}^{p}(t) := \sup_{s \in [0,t]}
    \wunorm{X(s)}^{p}$ we find that}
    \Expect[\wunorm{X}^{p}(\StopT{t})] &\le 
    \Expect[\wunorm{X(0)}^{p}]+\int_{0}^{t}
    C(1+E[\wunorm{X}^{p}(\StopT{s})] \, ds.
  \end{align*}
  Gronwall's inequality now implies that
  $\Expect[\wunorm{X}^{p}(\StopT{t})]$ is bounded in terms of the
  initial data and time $t$. By Fatou's lemma the claim follows by
  letting $\stopping \to \infty$.

  We next consider $\beta_{2} > 0$. Using Theorem~\ref{th:Epbound} we
  still have the bound \eqref{eq:Mqv} which yields
  \begin{align*}
    \Expect \left( [M]_{\StopT{t}}^{1/2} \right) &\le
    \int_{0}^{\StopT{t}} 
    C(1+\Expect[\wunorm{X(s)}^{p+1}]) \, ds 
    \le (e^{C\StopT{t}}-1)(\Expect[\wunorm{X(0)}^{p+1}]+1).
  \end{align*}
  where we similarly obtain a bound in terms of
  $\Expect[\wunorm{X(0)}^{p+1}]$.
\end{proof}

\subsection{Scaling}
\label{subsec:scaling}

We shall now regard the transport rates, the reaction rates, and the
magnitude of the state variables as problem parameters which may
induce a scale separation. Although a completely general multiscale
analysis is possible within the current framework, to fix our ideas
and in the interest of a transparent presentation, we consider a
concrete, but still quite general two-scale separation.

\begin{condition}[\textit{Scale separation}]
  \label{cond:scales}
  Let a \emph{scale vector} $\ScaleS \in \Realdom^{D}$ be given. The
  transport- and reaction rates are assumed to obey the scaling laws
  \begin{align}
    \label{eq:scaled}
    q_{ijk} x &= \epsilon^{-\mu_i} \bar{q}_{ijk} \ScaleS^{-1}x, \\
    \label{eq:scaler}
    u_r(x) &= \epsilon^{-\nu_r^{(1)}} \bar{u}_r(x) = 
    \epsilon^{-\nu_r^{(1)}-\nu_r^{(2)}} \bar{u}_r(\ScaleS^{-1}x) \\
    \intertext{for $i = 1 \ldots D$, $(j,k) = 1\ldots J$, and $r = 1
      \ldots R$. For the state variables we define}
    \label{eq:scalec}
    X_{i,\cdot}(t) &= \ScaleS_{i} \ScaleX_{i,\cdot}(t), \qquad
    \ScaleS_{i} = 1 \mbox{ or } \epsilon^{-1}.
  \end{align}
  where $\nu_r^{(1)}$ is the scaling of the rate (fast/slow) while
  $\nu_r^{(2)}$ follows from the number of species involved in
  transition $r$ such that $S_i = \epsilon^{-1}$. Let the complete
  scaling be
  $$\nu_r = \nu_r^{(1)} + \nu_r^{(2)}.$$
  The dynamics is considered for $t \in [0,T]$, $T = O(1)$ with
  respect to $\epsilon$. Also, all non-dimensionalized constants and
  propensities $\{\bar{q}_{ijk},\bar{u}_r(\cdot)\}$ are understood to
  be $O(1)$ with respect to $\epsilon$.
\end{condition}

It is possible to analyze also the general case where the species
scale differently in different voxels, i.e.~$X_{ij} = \ScaleS_{ij}
\ScaleX_{ij}$. However, this analysis is complicated by the fact that
the results then take place in a transient regime, and, in turn, this
regime is difficult to generally estimate.

We make a slight abuse of notation by employing $\ScaleS$ as if it was
the $D$-by-$D$ matrix $\diag(\ScaleS)$. Using a similar convention for
$\nu$ we may write \eqref{eq:scaler} in the compact form
\begin{align}
  \label{eq:scaler_compact}
  u(x) &= \epsilon^{-\nu} \bar{u}(\ScaleS^{-1}x).
\end{align}
To take a concrete example: the bimolecular reaction
$X+Y \to \emptyset$ at rate $kXY$ obeys \eqref{eq:scaler} with
$\nu_r = 0$ for $k \sim \epsilon$ and one of the species scaling
macroscopically as $\epsilon^{-1}$. If both species are macroscopic,
then instead $\nu_r = 1$ at the same scaling of the rate
$k \sim \epsilon$.

Following Condition~\ref{cond:scales} we thus divide the species into
two disjoint groups, $\Group_1$ and $\Group_2$, with
$|\Group_1|+|\Group_2| = D_1+D_2 = D$. Informally, we suppose that
species in low copy numbers are in $\Group_1$ and species in large
copy numbers are in $\Group_2$. Under an appropriate enumeration of
the species this implies the choice of scaling $\ScaleS_{i} = 1$ for
$i \in \Group_1 = \{1\ldots D_1\}$ and $= \epsilon^{-1}$ for $i \in
\Group_2 = \{D_1+1\ldots D\}$ in \eqref{eq:scalec}. Following this
ordering we also write $\wvec = [\wvec_1; \, \wvec_2]$ and $\stoich =
[\stoich^{(1)};\,\stoich^{(2)}]$, where $\wvec_i \in \Realdom_{\ge
  1}^{D_i}$ and $\stoich^{(i)} \in \Realdom^{D_i \times R}$ for $i \in
\{1,2\}$.

We find from \eqref{eq:RDMEPoissrepr} the governing equation
\begin{align}
  \label{eq:RDMEPoissreprS}
  \ScaleX_{ij}(t) = \ScaleX_{ij}(0) &-
  \sum_{r = 1}^{R} \ScaleS_{i}^{-1} \stoich_{ri} \Pi_{rj}  
  \left( \int_{0}^{t} 
  \VoxVol_j \epsilon^{-\nu_{r}} 
  \bar{u}_r (\VoxVol_j^{-1} \ScaleX_{\cdot,j}(s)) \, ds \right) \\
  \nonumber
  &-\sum_{k = 1}^{J} \ScaleS_{i}^{-1} \Pi_{ijk}' 
  \left( \int_{0}^{t}
  \epsilon^{-\mu_i} \bar{q}_{ijk} \ScaleX_{ij}(s)  \, ds \right) \\
  \nonumber
  &+\sum_{k = 1}^{J} \ScaleS_{i}^{-1} \Pi_{ikj}'
  \left( \int_{0}^{t}
  \epsilon^{-\mu_i} \bar{q}_{ikj} \ScaleX_{ik}(s) \, ds \right).
\end{align}

For the existence of scale separation it is critical to find
conditions such that according to some weight-vector $\lvec$,
$\lunorm{\ScaleX(t)}$ for $t \in [0,T]$ remains $O(1)$ whenever
$\lunorm{\ScaleX(0)}$ is $O(1)$, assuming that $T$ and $\lvec$ both
are $O(1)$ with respect to $\epsilon$. Unfortunately, the assumptions
and analysis in \S\ref{subsec:RDMEstab} all concerned the unscaled
variable $X(t)$, which is now assumed to be $O(\epsilon^{-1})$. In
fact, it is not difficult to see that with, say,
$\bar{\lvec} := \ScaleS\wvec$ replacing $\wvec$ throughout
Assumption~\ref{ass:rbound}, and requiring that all constants be
independent of $\epsilon$, the results in \S\ref{subsec:RDMEstab} are
straightforwardly translated into bounds in terms of the
$\bar{\lvec}$-norm of $\ScaleX(t)$. Since this is just the
$\wvec$-norm of $X(t)$ itself, however, it scales as
$O(\epsilon^{-1})$. What is additionally required is that the
weight-vector $\lvec$ can be selected \emph{independently} of
$\epsilon$.

\begin{assumption}[\textit{Reaction regularity, scaled case}]
  \label{ass:rboundS}
  The previous assumption of density dependent propensities
  \eqref{eq:ddepend} is assumed to hold. We further assume the
  existence of a vector $\lvec \in \Realdom^D_{\ge 1}$,
  \emph{independent of $\epsilon$}, such that
  \begin{align}
    \label{eq:1bndS}
    -\lvect \ScaleS^{-1} \stoich &u(x) \le A+\alpha \lnorm{\ScaleS^{-1} x}, \\
    \label{eq:bndS}
    (-\lvect \ScaleS^{-1} \stoich)^{2} &u(x)/2 \le B+\beta_{1}
    \lnorm{\ScaleS^{-1} x}+\beta_{2}\lnorm{\ScaleS^{-1} x}^{2}, \\
    \label{eq:lipS}
    |\bar{u}_{r}(x)-\bar{u}_{r}(y)| &\le
    \LipPr_{r}(\Sstopping)\|x-y\|, \mbox{ for $r = 1\ldots R$, and $\lnorm{x}
    \vee \lnorm{y} \le \Sstopping$}.
  \end{align}
  All parameters $\{A,\alpha,B,\beta_{1},\beta_{2},L\}$ are assumed to
  be independent of $\epsilon$ and non-negative (with negative values
  allowed for $\alpha$).
\end{assumption}

Equipped with this assumption we revisit the regularity results of
\S\ref{subsec:RDMEstab}. To this end we consider a version of
$\Sspace{p}$ scaled with $\ScaleS$,
\begin{align}
  \SspaceS{p} &\equiv \left\{ \ScaleX(t,\Probelem): \begin{array}{l}
    \ScaleX(t) \in S^{-1} \Intdom_{+}^{D \times J}
    \mbox{ is $\Probfiltr_{t}$-adapted such that } \\
    \Expect[\sup_{t \in [0,T]} \lunorm{\ScaleX_{t}}^{p}] <
    \infty \mbox{ for } \forall T < \infty \end{array} \right\}, \\
  \intertext{where the scaled state space is just}
  &S^{-1} \Intdom_{+}^{D \times J} \equiv 
  [\Intdom_{+}^{D_1 \times J};\; \epsilon \Intdom_{+}^{D_2 \times J}],
\end{align}
and $\epsilon \Intdom_+ = \{0,\epsilon,2\epsilon,\ldots\}$.

\begin{theorem}[\textit{Regularity, scaled case}]
  \label{th:exist1}
  Under Condition~\ref{cond:scales}, Theorem~\ref{th:Epbound} and
  \ref{th:exist0} both hold with the new Assumption~\ref{ass:rboundS}
  replacing the previous Assumption~\ref{ass:rbound} and with
  $\SspaceS{p}$ replacing $\Sspace{p}$. In particular:
  \begin{enumerate}
  \item The constant $C$ in Theorem~\ref{th:Epbound} can be selected
    independently of $\epsilon$.

  \item If either $\beta_2 = 0$ and $\Expect[\lunorm{\ScaleX(0)}^{p}]$
    is $O(1)$ with respect to $\epsilon$, or $\beta_2 > 0$ and
    $\Expect[\lunorm{\ScaleX(0)}^{p+1}]$ is $O(1)$, then so is
    $\Expect[\sup_{s \in [0,t]} \lunorm{\ScaleX(s)}^{p}]$ for $t \in
    [0,T]$, $T = O(1)$ with respect to $\epsilon$.
  \end{enumerate}
\end{theorem}

The proof follows very closely the steps taken to arrive at
Theorem~\ref{th:exist0} and is therefore
omitted. Theorem~\ref{th:exist1} inherits from Theorem~\ref{th:exist0}
the poorer regularity when $\beta_2 > 0$. The predicted growth is then
$\exp(t \beta_2^{\VoxVol})$ where, as in \eqref{eq:Fdef},
$\beta_2^{\VoxVol} := \beta_2 m_\VoxVol^{-1} \bar{\VoxVol}_{M}^{-1}$
and is dependent on the mesh.

\subsection{Multiscale splittings}
\label{subsec:splitting}

We shall consider two multiscale splittings: one ``exact'' in
continuous time and one ``numerical'' in discrete time-steps of length
$h$.

Thus we firstly define $\ScaleZ$, for $i$ in $\Group_1$ and using that
$\ScaleS_i = 1$,
\begin{align}
  \label{eq:RDMEPoissrepr_ex1}
  \ScaleZ_{ij}(t) = \ScaleZ_{ij}(0) &-
  \sum_{r = 1}^{R} \stoich_{ri} \Pi_{rj}  
  \left( \int_{0}^{t} 
  \VoxVol_j \epsilon^{-\nu_{r}} 
  \bar{u}_r (\VoxVol_j^{-1} \ScaleZ_{\cdot,j}(s)) \, ds \right) \\
  \nonumber
  &-\sum_{k = 1}^{J} \Pi_{ijk}' 
  \left( \int_{0}^{t}
  \epsilon^{-\mu_i} \bar{q}_{ijk} \ScaleZ_{ij}(s)  \, ds \right) \\
  \nonumber
  &+\sum_{k = 1}^{J} \Pi_{ikj}'
  \left( \int_{0}^{t}
  \epsilon^{-\mu_i} \bar{q}_{ikj} \ScaleZ_{ik}(s) \, ds \right),
  \intertext{while for $i$ in $\Group_2$, $\ScaleS_i = \epsilon^{-1}$
    and the Poisson process is approximated by a deterministic
    process,}
  \label{eq:RDMEPoissrepr_ex2}
  \ScaleZ_{ij}(t) = \ScaleZ_{ij}(0) &-
  \sum_{r = 1}^{R} \epsilon \stoich_{ri} 
  \left( \int_{0}^{t} 
  \VoxVol_j \epsilon^{-\nu_{r}} 
  \bar{u}_r (\VoxVol_j^{-1} \ScaleZ_{\cdot,j}(s)) \, ds \right) \\
  \nonumber
  &-\sum_{k = 1}^{J} \epsilon
  \left( \int_{0}^{t}
  \epsilon^{-\mu_i} \bar{q}_{ijk} \ScaleZ_{ij}(s)  \, ds \right) \\
  \nonumber
  &+\sum_{k = 1}^{J} \epsilon
  \left( \int_{0}^{t}
  \epsilon^{-\mu_i} \bar{q}_{ikj} \ScaleZ_{ik}(s) \, ds \right).
\end{align}
In general, there is no guarantee that $\ScaleZ(t)$ remains positive
even when $\ScaleX(t)$ is a conservative chain. For example, the
presence of a dimerization reaction, say, $A+A \to B$ at rate $A(A-1)$
can reach negative values of $B$ when $A$ is approximated by a
continuous variable. In this example one can avoid this problem by
reinterpreting the rate as $A(A-1) \vee 0$. In what follows we will
for simplicity assume that all models are conservative and remain in
the non-negative orthant, presumably after employing some kind of
limiters on the rates.

To see how a result similar to Theorem~\ref{th:exist1} might be
obtained for the new process $\ScaleZ(t)$, we start anew from Dynkin's
formula, appropriately modified for the semi-continuous setting. We
find
\begin{align}
  &\Expect\left[\lunorm{\ScaleZ(\StopT{t})}^{p}\right] = 
  \Expect\left[\lunorm{\ScaleZ(0)}^{p}\right]+
  \Expect\left[\int_0^{\StopT{t}} G(\ScaleZ(s)) \, ds \right],
  \intertext{where}
  \label{eq:Gdef}
  G(Z) &\equiv \sum_{j = 1}^J \sum_{r = 1}^{R} 
  \VoxVol_j u_r(\VoxVol_j^{-1}\ScaleS Z_{\cdot,j}) 
  \left[ \left( z-\lvect_1 \stoich_{r}^{(1)} \right)^{p}-
    z^{p}-p\lvect_2\epsilon\stoich_r^{(2)}z^{p-1}\right],
\end{align}
and where for brevity $z \equiv \lunorm{Z}$ (compare
\eqref{eq:Fdef}). Using
Lemma~\ref{lem:diffbound}~\eqref{eq:signed_diffbound} we find
{\small \begin{align}
    \label{eq:prelGbound}
  G(Z) &\le \sum_{j = 1}^J \sum_{r = 1}^{R} 
  \VoxVol_j u_r(\VoxVol_j^{-1}\ScaleS Z_{\cdot,j}) 
  \left[ -p\lvect \ScaleS^{-1}\stoich_{r}z^{p-1}+
    C_p/2 \left( -\lvect_1 \stoich_{r}^{(1)} \right)^2
    \left[z^{p-2}+C_{\stoich^{(1)}}^{p-2}\right] \right],
\end{align} }
where $C_p$ and $C_{\stoich^{(1)}}$ are defined below
\eqref{eq:Fbound}. The goal here is to obtain a bound $G(Z) \le
C(1+z^p)$ (compare \eqref{eq:Fbound}--\eqref{eq:momest}) and it is not
difficult to see what assumption is required.

\begin{assumption}[\textit{Reaction regularity, semi-continuous case}]
  \label{ass:rboundSZ}
  In Assumption~\ref{ass:rboundS}, replace \eqref{eq:bndS} with
  \begin{align}
    \label{eq:bndSZ}
    \left( -\lvect_1 \stoich^{(1)} \right)^2 u(x)/2 \le 
    B+\beta_1\lnorm{\ScaleS^{-1}x}+\beta_2\lnorm{\ScaleS^{-1}x}^{2}.
  \end{align}
\end{assumption}

This assumption can be understood as firstly, a signed bound
\eqref{eq:1bndS} on the drift-part for the fully coupled system, and
secondly, the extra assumption due to stochasticity \eqref{eq:bndSZ},
which here applies only to $i \in \Group_1$, that is, to the
stochastic part.

Using this in \eqref{eq:prelGbound} we find (compare
\eqref{eq:Fbound})
\begin{align}
  \label{eq:Gbound}
  G(Z) &\le p(A\Vol+\alpha z) z^{p-1}+
  C_p(B\Vol+\beta_{1} z+
  \beta_{2}^\VoxVol z^{2})
  (z^{p-2}+C_{\stoich^{(1)}}^{p-2}),
\end{align}
and following the steps in the proofs of Theorems~\ref{th:exist0} and
\ref{th:exist1} we obtain after some work the following result.

\begin{theorem}[\textit{Regularity, semi-continuous case}]
  \label{th:exist2}
  The statement of Theorem~\ref{th:exist1} applies also to the
  approximating process $\ScaleZ(t)$ with
  Assumption~\ref{ass:rboundSZ} taking the role of
  Assumption~\ref{ass:rboundS}. The existence of solutions then
  concerns the semi-continuous space $\SspaceZ{p}$ and we note that
  the remark following Theorem~\ref{th:exist1} concerning the
  dependence on the mesh regularity remains valid since
  $\beta_2^{\VoxVol}$ is present in \eqref{eq:Gbound}.
\end{theorem}

In practice, a numerical method is required to simulate $\ScaleZ$. The
most straightforward way is to evolve the stochastic and deterministic
parts in different steps, introducing a new process $\ScaleY$ which
approximates $\ScaleZ$. Following the partition of unity idea in
\cite{jsdesplit} we define the kernel step function
\begin{align}
  \sigma_{h}(t) &= 1-2 \left( \lfloor t/(h/2) \rfloor
  \mbox{{\footnotesize  mod }} 2 \right).
\end{align}
Then for $i$ in $\Group_1$,
\begin{align}
  \nonumber
  \ScaleY_{ij}(t) = \ScaleY_{ij}(0) &-
  \sum_{r = 1}^{R} \stoich_{ri} \Pi_{rj}  
  \left( \int_{0}^{t} 
  (1+\sigma_h(s)) \VoxVol_j \epsilon^{-\nu_{r}} 
  \bar{u}_r (\VoxVol_j^{-1} \ScaleY_{\cdot,j}(s)) \, ds \right) \\
  \label{eq:RDMEPoissrepr_num1}
  &-\sum_{k = 1}^{J} \Pi_{ijk}' 
  \left( \int_{0}^{t}
  (1+\sigma_h(s)) \epsilon^{-\mu_i} \bar{q}_{ijk} \ScaleY_{ij}(s)  \, ds \right) \\
  \nonumber
  &+\sum_{k = 1}^{J} \Pi_{ikj}'
  \left( \int_{0}^{t}
  (1+\sigma_h(s)) \epsilon^{-\mu_i} \bar{q}_{ikj} \ScaleY_{ik}(s) \, ds \right),
  \intertext{and for $i$ in $\Group_2$,}
  \nonumber
  \ScaleY_{ij}(t) = \ScaleY_{ij}(0) &-
  \sum_{r = 1}^{R} \epsilon \stoich_{ri} 
  \left( \int_{0}^{t} 
  (1-\sigma_h(s)) \VoxVol_j \epsilon^{-\nu_{r}} 
  \bar{u}_r (\VoxVol_j^{-1} \ScaleY_{\cdot,j}(s)) \, ds \right) \\
  \label{eq:RDMEPoissrepr_num2}
  &-\sum_{k = 1}^{J} \epsilon
  \left( \int_{0}^{t}
  (1-\sigma_h(s)) \epsilon^{-\mu_i} \bar{q}_{ijk} \ScaleY_{ij}(s)  \, ds \right) \\
  \nonumber
  &+\sum_{k = 1}^{J} \epsilon
  \left( \int_{0}^{t}
  (1-\sigma_h(s)) \epsilon^{-\mu_i} \bar{q}_{ikj} \ScaleY_{ik}(s) \, ds \right).
\end{align}

For regularity we start anew from the semi-continuous Dynkin's
formula,
\begin{align}
  &\Expect\left[\lunorm{\ScaleY(\StopT{t})}^{p}\right] = 
  \Expect\left[\lunorm{\ScaleY(0)}^{p}\right]+
  \Expect\left[\int_0^{\StopT{t}} H(\ScaleY(s),s) \, ds \right],
  \intertext{where this time}
  \nonumber
  H(Y,s) &\equiv \sum_{j = 1}^J \sum_{r = 1}^{R} 
  (1+\sigma_h(s))\VoxVol_j u_r(\VoxVol_j^{-1}\ScaleS Y_{\cdot,j}) 
  \left[ \left( y-\lvect_1 \stoich_{r}^{(1)} \right)^{p}-
    y^{p}\right] \\
  \label{eq:Hdef}
  &\phantom{\equiv}
  +\sum_{j = 1}^J \sum_{r = 1}^{R} 
  (1-\sigma_h(s))\VoxVol_j u_r(\VoxVol_j^{-1}\ScaleS Y_{\cdot,j})
   \left[ -p\lvect_2\epsilon\stoich_r^{(2)}y^{p-1}\right],
\end{align}
and where as before $y \equiv \lunorm{Y}$. This leads us to

\begin{assumption}[\textit{Reaction regularity, split-step case}]
  \label{ass:rboundSY}
  Besides the modification of Assumption~\ref{ass:rboundSZ}, in
  Assumption~\ref{ass:rboundS}, additionally replace \eqref{eq:1bndS}
  with
  \begin{align}
    \label{eq:1bndSY}
    \max \left(-\lvect_1 \stoich^{(1)} u(x),
    -\lvect_2 \epsilon \stoich^{(2)} u(x) \right) &\le 
    A+\alpha\lnorm{\ScaleS^{-1}x}.
  \end{align}
\end{assumption}
In other words, \eqref{eq:1bndSY} bounds the drift of the stochastic
and continuous parts individually, while as before \eqref{eq:bndSZ} is
employed to bound the quadratic variation of the stochastic part
alone.

Following again the steps in the previous proofs we obtain

\begin{theorem}[\textit{Regularity, split-step case}]
  \label{th:exist3}
  Theorem~\ref{th:exist2} applies also to the approximating process
  $\ScaleY(t)$ under Assumption~\ref{ass:rboundSY}. The resulting
  \textit{a priori} bound is uniform with respect to both $\epsilon$
  and $h$ provided the initial data is.
\end{theorem}

The approximation $\ScaleX \approx \ScaleZ$ gives rise to a
\emph{multiscale error}, whereas $\ScaleZ \approx \ScaleY$ induces a
\emph{splitting error}. Quite generally, any practical numerical
method relies on this very structure in $\ScaleX \approx \ScaleY$.
Insight into the nature of the total error thus follows from a
consistent analysis of both approximations. This is the purpose with
the next section.


\section{Error analysis}
\label{sec:analysis}

We present in this section the error analysis of the two
approximations
\eqref{eq:RDMEPoissrepr_ex1}--\eqref{eq:RDMEPoissrepr_ex2} and,
respectively,
\eqref{eq:RDMEPoissrepr_num1}--\eqref{eq:RDMEPoissrepr_num2}. Theorems~\ref{th:exist1},
\ref{th:exist2}, and \ref{th:exist3} assert that all processes are
uniformly stable in finite time. By the Lax principle the task has
therefore been reduced to an investigation of the degree of
consistency of the two approximations. Preliminary lemmas for this are
discussed in \S\ref{subsec:prelest}, followed by the actual error
analysis in \S\ref{subsec:msconv}--\ref{subsec:splitconv}. In order
not to lose the oversight, some material heavily relied upon are
developed separately in Appendix~\ref{sec:mserr} and \ref{sec:sserr}.

\subsection{Preliminary estimates}
\label{subsec:prelest}

Intuitively, the same version of a Poisson process evaluated at two
different operational times should enjoy a bounded difference,
provided of course the times themselves are bounded in some suitable
sense. A precise formulation of this property is related to
\emph{Doob's optional sampling theorem} \cite[Theorem 17,
  Chap.~I.2]{protterSDE} and has only just recently been investigated
\cite{AGangulyMMS2014, tauAnderson} for the $L^1$-norm, and in
\cite{jsdesplit} for the $L^2$-norm.

\begin{lemma}
  \label{lem:PiStop}
  Let $\Pi$ be a unit-rate $\Probfiltr_{t}$-adapted Poisson process,
  and let $T$ be a bounded stopping time. Then
  \begin{align}
    \label{eq:PiStop1}
    \Expect[\Pi(T)] &= \Expect[T], \\
    \label{eq:PiStop2}
    \Expect[\Pi^2(T)] &= 2\Expect[\Pi(T)T]-\Expect[T^{2}]+\Expect[T].
  \end{align}
\end{lemma}

\begin{proof}
  Let $\tilde{\Pi}(t) := \Pi(t)-t$ be the compensated process. This is
  a martingale and the sampling theorem implies
  $\Expect[\tilde{\Pi}(T)] = 0$, which is \eqref{eq:PiStop1}. The
  quadratic variation is $[\tilde{\Pi}]_{t} = \Pi(t)$ and hence $Z(t)
  := \tilde{\Pi}^2(t)-\Pi(t)$ is a local martingale. Since
  $\Expect[\sup_{s \le t} Z(s)] < \infty$ for bounded $t$, it is
  actually a martingale and the sampling theorem now yields
  $\Expect[Z(T)] = 0$, or,
  \begin{align*}
    0 &= \Expect[\Pi^2(T)-2\Pi(T)T+T^{2}-\Pi(T)],
  \end{align*}
  which is \eqref{eq:PiStop2}.
\end{proof}

\begin{lemma}
  \label{lem:Stopping}
  Let $\Pi$ be a unit-rate $\Probfiltr_{t}$-adapted Poisson process,
  and let $T_1$, $T_2$ be bounded stopping times. Then
  \begin{align}
    \label{eq:Stopping1}
    \Expect[|\Pi(T_2)-\Pi(T_1)|] &= \Expect[|T_2-T_1|], \\
    \label{eq:Stopping2}
    \Expect[(\Pi(T_2)-\Pi(T_1))^2] &=
    2\Expect[|\Pi(T_2)-\Pi(T_1)|(T_1 \vee T_2)] \\
    \nonumber
    &\hphantom{=} -\Expect[|T_2^{2}-T_1^{2}|]+\Expect[|T_2-T_1]].
  \end{align}
\end{lemma}

\begin{proof}
  Assume first that $T_2 \ge T_1$. We get from
  Lemma~\ref{lem:PiStop}~\eqref{eq:PiStop1}
  \begin{align*}
    \Expect[\Pi(T_2)-\Pi(T_1)] &= \Expect[T_2-T_1].
  \end{align*}
  For general stopping times $S_1$, $S_2$, say, not necessarily
  satisfying $S_2 \ge S_1$, \eqref{eq:Stopping1} now follows upon
  substituting $T_1 := S_1 \wedge S_2$ and $T_2 := S_1 \vee S_2$ into
  this equality.

  Next put $X := \Expect[(\Pi(T_2)-\Pi(T_1))^2]$ and assume again that
  $T_2 \ge T_1$. We get
  \begin{align*}
    X &= \Expect[\Pi(T_2)^2 +
        \Pi(T_1)^2 - 2\Pi(T_1)\Pi(T_2)] \\
      &= \Expect[\Pi(T_2)^2 + \Pi(T_1)^2] -
        2\Expect[\Pi(T_1)\Expect[\Pi(T_2)|\Probfiltr_{T_1}]]. \\
    \intertext{To evaluate the iterated expectation note that}
    \Expect[\tilde{\Pi}(T_2)|\Probfiltr_{T_1}] &= \tilde{\Pi}(T_1)
    \Longrightarrow
    \Expect[\Pi(T_2)|\Probfiltr_{T_1}] =
    \Pi(T_1)-T_1+E[T_2|\Probfiltr_{T_1}]. \\
    \intertext{Hence,}
    \Expect[\Pi(T_1)\Expect[\Pi(T_2)|\Probfiltr_{T_1}]] &= 
    \Expect[\Pi(T_1)^2]-
    \Expect[\Pi(T_1)T_1]+\Expect[\Pi(T_1)T_2], \\
    \intertext{and we thus find that}
    X &= \Expect[\Pi(T_2)^2 - \Pi(T_1)^2]
        +2\Expect[\Pi(T_1)T_1]-2\Expect[\Pi(T_1)T_2]. \\
    \intertext{Applying Lemma~\ref{lem:PiStop}~\eqref{eq:PiStop2}
      twice yields finally}
    X &= 2\Expect[(\Pi(T_2)-\Pi(T_1))T_2]-\Expect[T_2^{2}-T_1^{2}]+
    \Expect[T_2-T_1].
  \end{align*}
  For general stopping times $S_1$, $S_2$, \eqref{eq:Stopping2} now
  follows as before upon substituting $T_1 := S_1 \wedge S_2$ and $T_2
  := S_1 \vee S_2$.
\end{proof}

\begin{remark}
  We will use Lemma~\ref{lem:Stopping} in the following form. Assuming
  $T_1 \vee T_2$ has been bounded \textit{a priori} by some value $B$
  we get by combining \eqref{eq:Stopping1} with \eqref{eq:Stopping2}
  that
  \begin{align}
    \label{eq:finalStop}
    \Expect[(\Pi(T_2)-\Pi(T_1))^2] &\le
    (2B+1)\Expect[|T_2-T_1|].
  \end{align}
  Let $\Probfiltr_t$ be the filtration adapted to $\tilde{\Pi}_r, r =
  1 \ldots R$. Then for a fixed $t$, $ T_r(t) = \int_{0}^{t} w_r(X(s))
  \, ds$ is a stopping time \cite[Lemma~3.1]{tauAnderson} with respect
  to $$\tilde{\Probfiltr}^r_u = \sigma\{\Pi_r(s), s \in [0,u];\;
  \Pi_{k \not = r}(s), s\in [0,\infty]\}.$$ Intuitively, as $X(t) =
  \sum_r \Pi_r(T_r(t)) \stoich_r$, the event $\{T_r(t) < u\}$ depends
  on $\Pi_r$ during $[0,u]$ and on all other processes $\{\Pi_k, k\neq
  r\}$ during $[0,\infty)$. However, as $\Pi_r$, $r=1\ldots R$ are
    independent, $\Pi_r(t) - t$ is still a martingale with respect to
    $\tilde{\Probfiltr}^r_u$ (and not only with respect to
    $\Probfiltr^r_u = \sigma\{\Pi_r(s), s \in [0,u] \}$). Hence we can
    apply the stopping time theorems to $T_r(t)$ and the previous
    lemmas therefore apply. The result stays true for the
    approximating process $Z$ (and later $Y$). Hence, given the bound
  \begin{align}
    \int_{0}^{t}w_r(X(s)) \, ds \vee 
    \int_{0}^{t}w_r(Y(s)) \, ds \le B
  \end{align}
  we get from \eqref{eq:finalStop} that
  \begin{align}
    \label{eq:poisson_rmk}
    \nonumber
    \Expect\left[\left(\Pi_r\left(\int_{0}^{t}w_r(X(s)) \, ds\right)-
      \Pi_r\left(\int_{0}^{t}w_r(Y(s)) \, ds\right)\right)^2\right] \\
    \le
    (2B+1)\Expect\left[\left|\int_{0}^{t}w_r(X(s)) \, ds-
      \int_{0}^{t}w_r(Y(s)) \, ds\right|\right].
  \end{align}
\end{remark}

\subsection{Multiscale convergence}
\label{subsec:msconv}

This section develops a bound for the multiscale error made in the
approximation $\ScaleX \approx \ScaleZ$. Throughout \S\ref{sec:jsdes},
a certain weighted norm which greatly simplified the theory was
used. However, in the present case of bounding errors we are
interested in the more conventional $L^2$-norm,
\begin{align}
  \pnorm{\ScaleX(t)}^2 \equiv \sum_{j=1}^J
  \pnorm{\ScaleX(t)_{\cdot,j}}^2 = \sum_{j=1}^J \sum_{i = 1}^D
  \ScaleX(t)_{ij}^2,
\end{align}
where, for convenience, from now on we shall write $\left\| \cdot \right\|$
instead of $\pnorm{\cdot}$.

Let $\Sstopping > 0$ and define the joint stopping time
\begin{align}
  \label{eq:StopT}
  \tau_{\Sstopping} &:= \inf_s\{\|\ScaleX(s)\| \vee
  \|\ScaleZ(s)\| \vee \|\ScaleY(s)\| > \Sstopping\},
  \mbox{ and put } \StopT{t} := \tau_{\Sstopping} \wedge t.
\end{align}
Recall the stopping time $T_r(t)$ from the remark after
Lemma~\ref{lem:Stopping}. Clearly, for any fixed $t$, $T_r(\StopT{t})$
is still a stopping time.

The first step in the analysis is to split the error in one part which
is bounded and one part which is not,
\begin{align}
  \Expect \left[ \|\ScaleX(t)-\ScaleZ(t)\|^2 \right] = 
  \Expect \left[ \|\ScaleX(t)-\ScaleZ(t)\|^2 1_{t > \StopT{t}}  \right]
  +\Expect \left[ \|\ScaleX(t)-\ScaleZ(t)\|^2 1_{t \leq \StopT{t}} \right].
\end{align}
The requirement to be able to control the contribution from the
non-bounded part motivates the following lemma:
\begin{lemma}
  \label{lem:Pbound}
  For any $p>1$, there exists a constant $K_p$ independent from
  $\epsilon$ and $h$ such that
  \begin{align}
    \label{eq:KpX}
    \Expect \left[\|\ScaleX\|^2 1_{t > \StopT{t}}\right] &\le 
    K_p \Sstopping^{-p/2}, \\
    \label{eq:KpZ}
    \Expect \left[\|\ScaleZ\|^2 1_{t > \StopT{t}}\right] &\le 
    K_p \Sstopping^{-p/2}, \\   
    \label{eq:KpY} 
    \Expect \left[\|\ScaleY\|^2 1_{t > \StopT{t}}\right] &\le
    K_p \Sstopping^{-p/2}.
  \end{align}
\end{lemma}

\begin{proof}
  Theorem~\ref{th:exist1} yields
  \begin{align*}
	\Expect\left[\lunorm{\ScaleX(t)}^{4}\right] &\le
        \left( \Expect\left[\lunorm{\ScaleX(0)}^{4}\right] +1\right)\exp(Ct)-1.
  \end{align*}
  Since $\left\|\cdot\right\|$ and $\lunorm{\cdot}$ are equivalent
  bounds we have an \textit{a priori} bound
  \begin{align*}
    \Expect \left[ \|\ScaleX(t)\|^4\right] &\le B(t),
  \end{align*}
  with $B(t)$ independent from $\epsilon$. By Cauchy-Schwartz's
  inequality,
  \begin{align*}
   \Expect \left[\|\ScaleX(t)\|^2 1_{t > \StopT{t}}\right]
    &\leq \Expect \left[\|\ScaleX(t)\|^4 \right]^{1/2}\Prob[t > \StopT{t}]^{1/2}.
  \end{align*}
  Using that
  \begin{align*}
    \Prob[t > \StopT{t}]  \leq \Prob[\sup_{s \in [0,t]}
    \|\ScaleX_{s}\| > \Sstopping] + \Prob[\sup_{s \in [0,t]}
    \|\ScaleZ_{s}\| > \Sstopping] + \Prob[\sup_{s \in [0,t]}
    \|\ScaleY_{s}\| > \Sstopping],
  \end{align*}
  we find from Markov's inequality the bound
  \begin{align*}
    \Prob[t > \StopT{t}] \times \Sstopping^{p} &\leq 
    \Expect\left[ \left(\sup_{s\in [0,t]} \|\ScaleX(s)\|\right)^p
      \right] +  \Expect\left[ \left(\sup_{s\in [0,t]}
      \|\ScaleY(s)\|\right)^p \right] \\
    &+ \Expect\left[ \left(\sup_{s\in [0,t]} \|\ScaleZ(s)\|\right)^p \right].
  \end{align*}
  Using the second part of Theorem~\ref{th:exist1} and the equivalence
  of norms, it is possible to bound the first term on the right
  independently from $\epsilon$ and $h$. Reasoning similarly for the
  terms depending on $\ScaleY$ and $\ScaleZ$ we get the stated result.
\end{proof}

To formulate the main result of this section we let
\begin{align}
  R(\Group_1) &:= \{ r; \; \exists i  \in  
                \Group_1 \text{ such that } \stoich_{ri} \neq 0 \},
\end{align}
and the analogous definition for $R(\Group_2)$. In words,
$R(\Group_1)$ contains the reactions which affect any species $i \in
\Group_1$. We additionally define the two effective exponents
\begin{align}
  \label{eq:udef}
  u &= \min_{r \in R(\Group_1)} -\nu_r \wedge \min_{i\in
      \Group_1} -\mu_i, \\
  \label{eq:vdef}
  v &= 1+\min_{r \in R(\Group_2)} -\nu_r \wedge \min_{i \in \Group_2} -\mu_i.
\end{align}
Note that, if the transport rates do not scale with $\epsilon$, we
generally get $u \le 0$ and $v \le 1$.

\begin{theorem}[\textit{Multiscale error, bounded version}]
  \label{th:ScalingErrorBounded}
  Under the scale separation Condition~\ref{cond:scales}, the
  regularity Assumptions~\ref{ass:rboundS} and \ref{ass:rboundSZ}, and
  assuming also that $\ScaleZ$ and $\ScaleX$ are uniformly bounded
  with respect to $\epsilon$ by some $\Sstopping$, then whenever
  $u \geq 0$, $v \geq 0$ it holds that
  \begin{align}
    \Expect[\|\ScaleZ(t)-\ScaleX(t)\|^2] = O(\epsilon^{1+v} +
    \epsilon^{1/2 + v/2 + u}).
  \end{align}
\end{theorem}
\begin{proof}
  First notice that, since the processes are uniformly bounded with
  respect to $\epsilon$, so is $\LipPr_r$. Thus according to
  Lemma~\ref{lem:PDMPCVSquare},
  \begin{align*}
    \Expect[\|\ScaleZ(\StopT{t})-\ScaleX(\StopT{t})\|^2] &\lec
    A+B \int_0^t \Expect[\|\ScaleZ(\StopT{s})-\ScaleX(\StopT{s})\|^2 ] \, ds+
   C \int_0^t \Expect[\|\ScaleZ(\StopT{s})-\ScaleX(\StopT{s})\| ] \, ds,
  \end{align*}
  with
  \begin{align*}
    A &= O(\epsilon^{1+v}), \;
    B = O(\epsilon^{2v}), \;
    C = O(\epsilon^{u}).
  \end{align*}
  Similarly, according to
  Lemma~\ref{lem:PDMPCV}, $$\Expect[\|\ScaleX(\StopT{t})-\ScaleZ(\StopT{t})\|]
  \lec D + E \int_{0}^{t} \Expect[\|\ScaleX(\StopT{s}) -
    \ScaleZ(\StopT{s})\|] \, ds,$$ where
  \begin{align}
    D &= O(\epsilon^{1/2+v/2}), \; 
    E = O(\epsilon^{u}).
  \end{align}
  Thus using the Gronwall inequality we find firstly,
  \begin{align*}
    \Expect[\|\ScaleX(\StopT{t})-\ScaleZ(\StopT{t})\|] &=
    O(\epsilon^{1/2+v/2}). \\
    \intertext{Using this and Gronwall's inequality a second time gives}
    \Expect[\|\ScaleZ(\StopT{t})-\ScaleX(\StopT{t})\|^2]  &= O(\epsilon^{1 + v} + \epsilon^{1/2 + v/2 + u}).
  \end{align*}
  Suppose for the moment that $\ScaleY$ is uniformly bounded by
  $\Sstopping$ with respect to $\epsilon$ and $h$.  As the processes
  are bounded by $\Sstopping$,
  $\Expect[\|\ScaleX(\StopT{t})-\ScaleZ(\StopT{t})\|^2] =
  \Expect[\|\ScaleX(t)-\ScaleZ(t)\|^2]$ and we get the stated result.

  The extra assumption that $\ScaleY$ is uniformly bounded can easily
  be removed by changing the definition of $\tau$ in \eqref{eq:StopT}
  into
  \begin{align*}
    \tau_{\Sstopping} &:= \inf_s\{\|\ScaleX(s)\| \vee
    \|\ScaleZ(s)\| > \Sstopping\}.
  \end{align*}
\end{proof}

The two terms in the error bound can be interpreted as firstly, the
error introduced in the macro-species, $\epsilon^{1+v}$, and secondly,
the error made in the meso-species, $\epsilon^{1/2 + v/2 + u}$,
respectively.

In order to obtain a theorem also in the unbounded case, the growth of
the local Lipschitz constants has to be controlled, and so we make the
following convenient assumption:
\begin{assumption}
  \label{ass:lip}
  There exists $\lipa_1,\ldots,\lipa_R \geq 0$ such that $
  \LipPr_r(\Sstopping) \lec \Sstopping^{\lipa_r}$. Furthermore, we
  assume $\lipa_r = 0$ for each $r$ such that that $\nu_r = 0$. Hence
  the Lipschitz constants associated with these transitions are
  bounded independently from $\Sstopping$.
\end{assumption}

As in the appendix we use the notation ``$A \lec B$'' to indicate that
$A \le C B$ for some constant $C > 0$ which is $O(1)$ with respect to
$\epsilon$, $\Sstopping$, and $h$.

\begin{theorem}[\textit{Multiscale error}]
  \label{th:ScalingError}
  Under the scale separation Condition~\ref{cond:scales}, and under
  the regularity Assumptions~\ref{ass:rboundS}, \ref{ass:rboundSZ},
  and \ref{ass:lip}, and the additional conditions $u \geq 0$, $v > 0$,
  it holds that
  \begin{align}
    \Expect[\|\ScaleZ(t)-\ScaleX(t)\|^2] = O(\epsilon^{1+v -} +
    \epsilon^{1/2 + v/2 + u-}).
  \end{align}
\end{theorem}

\begin{proof}
  The proof here concerns the case $u > 0$. The special case from
  Assumption~\ref{ass:lip} where $\nu_r = 0$ and $a_r = 0$ for some
  $r$ (and thus $u = 0$) is similar but requires some cumbersome
  notation and is therefore omitted. Select
  $\Sstopping = \epsilon^{-\lipb}$ for some $\lipb>0$ and let $p>1$,
  $\lipa := \max_r \lipa_r$.  Following the same pattern as in the
  proof of Theorem~\ref{th:ScalingErrorBounded} we get
  $$\Expect[\|\ScaleZ(\StopT{t})-\ScaleX(\StopT{t})\|^2]  = O(\epsilon^{1 + v-(\lipa+1) \lipb} + \epsilon^{1/2 + v/2 + u-(3\lipa + 1) \lipb/2}).$$
  Thus using Lemma~\ref{lem:Pbound},
  \begin{align*}
    \Expect[\|\ScaleZ(t)-\ScaleX(t)\|^2] 
    &= \Expect[\|\ScaleZ(\StopT{t})-\ScaleX(\StopT{t})\|^2 1_{t\leq \StopT{t}}] + \Expect[\|\ScaleZ(t)-\ScaleX(t)\|^2 1_{t > \StopT{t}}] \\
    &\leq \Expect[\|\ScaleZ(\StopT{t})-\ScaleX(\StopT{t})\|^2] + 4 K_p \epsilon^{\lipb p/2} \\
    &=  O(\epsilon^{1 + v-(\lipa+1) \lipb} + \epsilon^{1/2 + v/2 + u -
      (3\lipa + 1) \lipb/2} +  \epsilon^{\lipb p/2}) .
  \end{align*}
  As $\lipb p/2$ can be made arbitrarily large while $(\lipa+1) \lipb$
  and $(3\lipa + 1) \lipb/2$ can be made arbitrarily close to $0$
  (i.e.~$\lipb \rightarrow 0$ and $p \rightarrow \infty$), we arrive
  at the stated bound.
\end{proof}

\begin{remark}
  It is possible to get a convergence result for the case $u = v =
  0$. However, in this case the error bound is of the form
  $O(\log(1/\epsilon)^{-\delta})$ and the dominating part can be traced back
  to Lemma~\ref{lem:Pbound}.
\end{remark}

\subsection{Splitting convergence}
\label{subsec:splitconv}

We next consider the error in the approximation
$\ScaleZ \approx \ScaleY$, that is, the splitting error. For this part
we are able to prove a somewhat weak error bound in the general case,
while the situation improves considerably if the processes are assumed
to be bounded \textit{a priori}.

\begin{theorem}[\textit{Splitting error, bounded version}]
  \label{th:spliterr2}
  Under the scale separation Condition~\ref{cond:scales}, the
  regularity Assumptions~\ref{ass:rboundSZ}, \ref{ass:rboundSY}, and
  assuming also that $\ScaleX$, $\ScaleZ$, and $\ScaleY$ are uniformly
  bounded with respect to $h$ and $\epsilon$ by $\Sstopping$, then
  whenever $u\geq 0$, $v \geq 0$ it holds that
  \begin{align}
    \Expect\left[\|\ScaleZ(t)-\ScaleY(t)\|^2\right] \le 
    O\left(h(\epsilon^{2u} + \epsilon^{u+v}) \right)  + O\left(h^2\epsilon^{2v}\right).
  \end{align}
\end{theorem}

\begin{proof}
  Using Lemma~\ref{lem:SplitStepCVSquare},
  \begin{align*}
    \Expect\left[\|\ScaleZ(\StopT{t})-\ScaleY(\StopT{t})\|^2 \right] &= 
	O(\epsilon^u) \int_0^t \Expect \left[\|\ScaleZ(\StopT{s})-\ScaleY(\StopT{s})\| \right] \, ds \\
	&+ O(\epsilon^{2v} )
        \int_0^t \Expect[\|\ScaleZ(\StopT{s})-\ScaleY(\StopT{s})\|^2] \, ds \\
        &+ O(\epsilon^uh) +  O(\epsilon^{2v} h^2 ).
  \end{align*}
  Using Lemma~\ref{lem:SplitStepCV} and the Gronwall inequality, one
  readily shows that
  \begin{align*}
    \Expect \left[ \|\ScaleZ(\StopT{t})-\ScaleY(\StopT{t})\| \right] &= 
    O\left( (\epsilon^u + \epsilon^v)h\right).
  \end{align*}
  Taken together we find
  \begin{align*}
    \Expect\left[\|\ScaleZ(t)-\ScaleY(t)\|^2\right] &\le O\left((\epsilon^{2u} + \epsilon^{u+v}) h\right) 
    + O\left(\epsilon^{2v} h^2\right) \\
                                                    &+ O(\epsilon^{2v} ) \int_0^t \Expect \left[\|  \ScaleZ(\StopT{s})- \ScaleY(\StopT{s})\|^2 \right] \, ds.
  \end{align*}
  Hence using the Gronwall inequality anew,
  \begin{align*}
    \Expect\left[\|\ScaleZ(t)-\ScaleY(t)\|^2 1_{t\leq \StopT{t}}\right] &= 
    O\left((\epsilon^{2u} + \epsilon^{u+v}) h\right)  + 
    O\left(\epsilon^{2v} h^{2}\right).
  \end{align*}
  Furthermore, as the processes are bounded by $\Sstopping$,
  $\Expect[\|\ScaleZ(\StopT{t})-\ScaleY(\StopT{t})\|^2] =
  \Expect[\|\ScaleZ(t)-\ScaleY(t)\|^2]$ and we get the stated result.
\end{proof}

As before one can appreciate the two terms of the error as the error
made in the meso-species, $(\epsilon^{2u} + \epsilon^{u+v})h$, and
$\epsilon^{2v} h^{2}$, the error introduced in the macro-species.

\begin{theorem}[\textit{Splitting error}]
  \label{th:spliterr1}
  Under the scale separation Condition~\ref{cond:scales}, and under
  the regularity Assumptions~\ref{ass:rboundSZ}, \ref{ass:rboundSY},
  and \ref{ass:lip}, and the additional conditions $u > 0$, $v > 0$,
  it holds that
  \begin{align}
    \lim_{h \to 0} \Expect\left[\|\ScaleZ(t)-\ScaleY(t)\|^2\right] = 0.
  \end{align}
\end{theorem}

\begin{proof}
  Following the same pattern as in the proof of the bounded version,
  it is easy to show that for each $\Sstopping$,
  $$ \Expect\left[\|\ScaleZ(t)-\ScaleY(t)\|^2 1_{t\leq \StopT{t}}\right] \xrightarrow[h \rightarrow 0]{} 0.$$
  We conclude the argument using Lemma~\ref{lem:Pbound}, which implies
  that
  $$\Expect\left[\|\ScaleZ(t)-\ScaleY(t)\|^2 1_{t\geq
      \StopT{t}}\right]\xrightarrow[\Sstopping \rightarrow \infty]{}
  0$$
  uniformly with respect to $h$.
\end{proof}

\begin{remark}
  Under the Assumptions of Theorem~\ref{th:spliterr1}, it is possible
  to get an error bound of the form
  $$\Expect\left[\|\ScaleZ(t)-\ScaleY(t)\|^2\right] \le O
  \left(\log(1/h)^{-\delta} \right),$$
  for any $\delta$ greater than some $\delta_0$. However, in this case
  the error can be traced to the unbounded part as covered by
  Lemma~\ref{lem:Pbound}.
\end{remark}


\section{Numerical examples}
\label{sec:examples}

We now proceed to illustrate our main findings through some
prototypical cases. An all-linear isomerization-type system is
investigated in \S\ref{subsec:isomerization} and a nonlinear catalytic
model in \S\ref{subsec:catalytic}.

In the experiments below we considered reactions taking place in a
one-dimensional geometry $[0,1)$ under periodic boundary
conditions. The geometry was discretized into 10 equally spaced
segments and a diffusion process implemented via the standard 2nd
order finite difference stencil, re-interpreted as linearly
dependent transition rates. As for the initial data, we let each
segment contain either 10 or 20 molecules for the mesoscopic
(discrete) species and $20\epsilon^{-1}$ or $10\epsilon^{-1}$ for
the macroscopic (continuous) species, respectively.

The exact dynamics \eqref{eq:RDMEPoissreprS} was simulated in an
operational time framework. Here we relied on an implementation of the
\emph{All Events Method} \cite{aem_proceeding}, essentially a spatial
extension of the \emph{Common Reaction Path Method}
\cite{sensitivitySSA} which evolves \eqref{eq:RDMEPoissreprS} using
separate Poisson processes for all events.

The multiscale approximation
\eqref{eq:RDMEPoissrepr_ex1}--\eqref{eq:RDMEPoissrepr_ex2} falls under
the scope of \emph{Piecewise Deterministic Markov Processes (PDMPs)}
for which accurate methods have been proposed \cite{hybridMarkov}. We
implemented this through the use of \emph{event-detection} in solvers
for Ordinary Differential Equations (ODEs). Notably, this allows for a
fully consistent coupling with \eqref{eq:RDMEPoissreprS} in
operational time.

Finally, the split-step approximation
\eqref{eq:RDMEPoissrepr_num1}--\eqref{eq:RDMEPoissrepr_num2} was
implemented. This is quite straightforward via the kernel step
function representation and executes very efficiently. The split-step
error is much more challenging to determine accurately than the
multiscale error is. In fact, on a predetermined grid in time the
split-step approximation $\ScaleY_{ij}$ in
\eqref{eq:RDMEPoissrepr_num1} was often found to be exactly equal to
the multiscale approximation $\ScaleZ_{ij}$ in
\eqref{eq:RDMEPoissrepr_ex1}, thus requiring many realizations for
even a very crude estimate.

We make repeated use of the estimator
\begin{align}
  \label{eq:mean}
  \Expect[(Y-X)(t)]^{2} &\approx M \equiv 
  \frac{1}{N}\sum_{i = 1}^{N} (Y-X)(t; \Probelem_{i})^{2},
  \intertext{for independent trajectories $(\Probelem_{i})$. A basic
    confidence interval is obtained by computing}
  \label{eq:std}
  S^{2} &\equiv \frac{1}{N-1}\sum_{i = 1}^{N}
  \left[(Y-X)(t; \Probelem_{i})^{2}-M \right]^{2},
\end{align}
such that the error in the estimator \eqref{eq:mean} is $\propto
S/\sqrt{N}$.

\subsection{Isomerization}
\label{subsec:isomerization}

We first consider the simple linear isomerization reaction pair,
\begin{align}
  A  & \xrightleftharpoons[k_b B]{k_a A} B.
\end{align}
In order for this example to develop a scale separation, for $A$, the
diffusion rate is set to $1/2$ in either direction and per molecule,
and for $B$ to $0$. By selecting $k_a = 1$ and $k_b = \epsilon$, a
scale separation occurs, with $A \sim 10$ and
$B \sim 10\epsilon^{-1}$. We may thus evolve the system by the
multiscale approximation
\eqref{eq:RDMEPoissrepr_ex1}--\eqref{eq:RDMEPoissrepr_ex2}, letting
$A$ remain discrete while $B$ is approximated with a continuous scaled
variable.

Although the unscaled system is closed, from the perspective of scale
separation the system scales unfavorably with $\epsilon$ and hence
falls under the scope of Theorem~\ref{th:ScalingError}. We have
$u = 0$ and $v = 1$ in \eqref{eq:udef}--\eqref{eq:vdef} and thus
expect a mean square error behaving like $O(\epsilon^2)$ for the
macroscopic species and $O(\epsilon)$ for the mesoscopic species. This
is verified in Figure~\ref{fig:iso1} where the multiscale error for
the two components is examined.

Since Theorem~\ref{th:spliterr2} is formally not applicable, the only
result valid is the guaranteed convergence of
Theorem~\ref{th:spliterr1}. Nevertheless, in Figure~\ref{fig:iso2} the
split-step error for the two species have been plotted separately. The
different terms of the error estimate in Theorem~\ref{th:spliterr2}
are clearly visible, suggesting that the uniform bounds on the
processes, as required by Theorem~\ref{th:spliterr2}, may in fact be
relaxed.

Convergence results similar to those of \cite{kurtz_multiscale} and
\cite{kurtz_multiscale2} are here consequences of
Theorem~\ref{th:ScalingError}, with the added benefit of an error
estimate. Indeed, Theorem~\ref{th:ScalingError} yields that the
difference between $\ScaleX$ and $\ScaleZ$ goes to 0 and the
convergence of $\ScaleZ$ is easy to study. Using
\eqref{eq:RDMEPoissrepr_ex1} and \eqref{eq:RDMEPoissrepr_ex2} for
voxel $j$ yields
\begin{align}
  \ScaleZ_{B,j}(t) &= \ScaleZ_{B,j}(0) + \epsilon \int_0^t k_a \ScaleZ_{A,j}(s) \, ds - \epsilon \int_0^t k_b \epsilon^{-1}\ScaleZ_{B,j}(s) \, ds \xrightarrow{\epsilon \rightarrow 0} \ScaleZ_{B,j}(0),
\end{align}
since $(k_a,k_b) = (1,\epsilon)$, and,
\begin{align}
  \nonumber
  \ScaleZ_{A,j}(t) \xrightarrow{\epsilon \rightarrow 0} 
  &\ScaleZ_{A,j}(0) + \Pi_{1,j} \left(\ScaleZ_{B,j}(0) t \right) - 
  \Pi_{2,j} \left(\int_0^t \ScaleZ_{A,j}(s) ds \right) \\
  &+ \sum_{k\in \{j-1,j+1\}} \Pi'_{A,k,j} \left( \int_0^t \ScaleZ_{A,k}(s)/2 \, ds \right) -\Pi'_{A,j,k} \left( \int_0^t \ScaleZ_{A,j}(s)/2 \, ds \right).
\end{align}
Hence for this simple system, the limit $\epsilon \to 0$ for $B$ is
trivial.

\begin{figure}
  \includegraphics{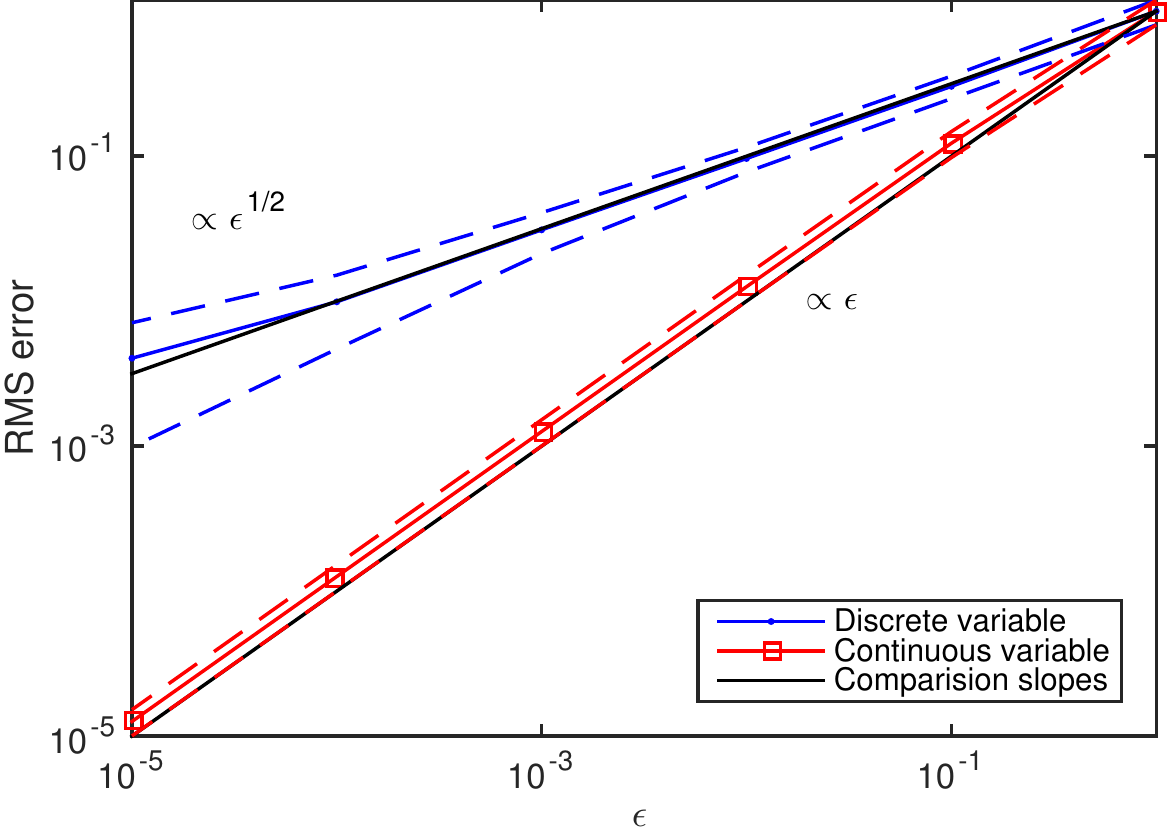}
  \caption{Multiscale error (isomerization): the root mean-square
    (RMS) error as a function of the scale separation $\epsilon$ for
    the two components $A$ (discrete) and $B$ (continuous).}
  \label{fig:iso1}
\end{figure}

\begin{figure}
  \includegraphics{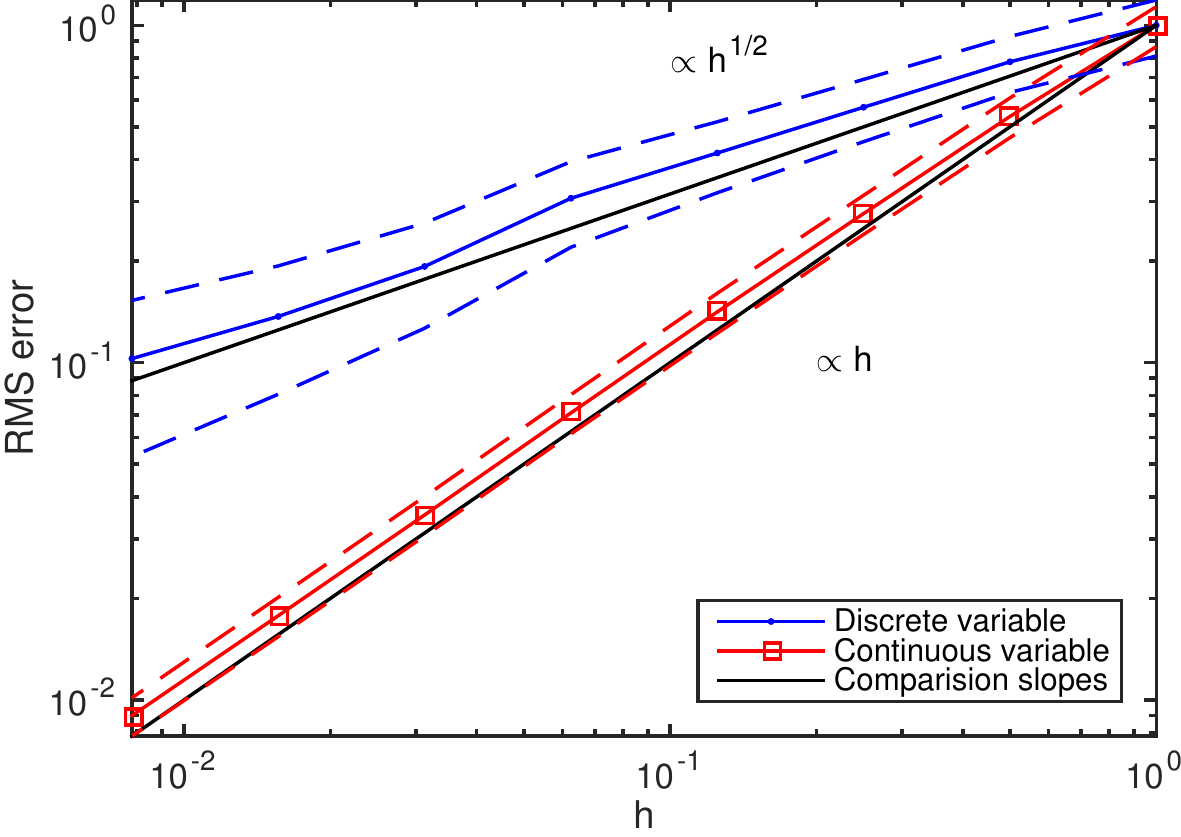}
  \caption{Split-step error (isomerization): the RMS error as a
    function of the split-step $h$ for the two components. Here
    $\epsilon = 10^{-1}$ was used (see Figure~\ref{fig:iso1}).}
  \label{fig:iso2}
\end{figure}

\subsection{Catalytic reactions}
\label{subsec:catalytic}

We consider the following pair of catalytic reactions:
\begin{align}
  \left. \begin{array}{rcl}
           A+B &\xrightarrow{kAB}& C+B \\
           C+D &\xrightarrow{kCD}& A+D \\
           B  & \xrightleftharpoons[k_d D]{k_b B}& D
         \end{array} \right\}
\end{align}
We assume that species $A$ and $C$ are abundant and
$O(\epsilon^{-1})$, and species $B$ and $D$ are $O(1)$. For the
diffusion we put $\sigma_{A,C} = \epsilon$ and $\sigma_{B,D} = 1$, and
for the rates $k = 0.01$ and $(k_b,k_d) = (1,0.9)$. The system so
defined is closed since there is no coupling from the macro-species to
the meso-species (take $\lvec = [1,1,1,1]^T$ in
Assumption~\ref{ass:rboundS}). This property carries over to the
multiscale and split-step approximations
(cf.~Assumptions~\ref{ass:rboundSZ} and \ref{ass:rboundSY}).

For the scale separation, we get the critical exponents $u = v = 0$
and Theorem~\ref{th:ScalingErrorBounded} predicts a slow convergence
of $O(\epsilon^{1/4})$ in the RMS sense. However, since the
meso-species do not depend on the macro-species the corresponding
error is in fact $0$. According to the discussion following the proof
of Theorem~\ref{th:ScalingErrorBounded}, the RMS is therefore
$O(\epsilon^{1/2})$ and is observed in the macroscopic species
only. By the same argument, and from the remark following the proof of
Theorem~\ref{th:spliterr2}, we predict that the RMS of the split-step
error is $O(h)$.

Experimental results verifying this are shown in Figure~\ref{fig:cat1}
for the multiscale error \emph{(``convergent scaling'')} and in
Figure~\ref{fig:cat2} for the split-step error.

Like in the previous example, convergence results similar to those of
\cite{kurtz_multiscale} and \cite{kurtz_multiscale2} are consequences
of Theorem~\ref{th:ScalingErrorBounded}.  This time,
\eqref{eq:RDMEPoissrepr_ex1} and \eqref{eq:RDMEPoissrepr_ex2} are
almost independent of $\epsilon$; only the diffusion for $A$ and $C$
depend on $\epsilon$ and, since $\sigma_{A,C} = \epsilon$, it vanishes
in the limit. For voxel $j$,
\begin{align}
  \ScaleZ_{A,j}(t) &\xrightarrow{\epsilon \rightarrow 0} -\int_0^t k \ScaleZ_{A,j}(s) \ScaleZ_{B,j}(s) \, ds  + \int_0^t k \ScaleZ_{C,j}(s) \ScaleZ_{D,j}(s) \, ds, \\
  \nonumber
  \ScaleZ_{B,j}(t) &= -\Pi_{1,j} \left( k_b \int_0^t \ScaleZ_{B,j}(s) \, ds \right) + \Pi_{2,j} \left( \int_0^t k_a \ScaleZ_{D,j}(s) \, ds \right) \\
  &+ \sum_{k\in \{j-1,j+1\}} \Pi'_{B,k,j} \left( \int_0^t \ScaleZ_{B,k}(s) \, ds \right) - \Pi'_{B,j,k} \left( \int_0^t \ScaleZ_{B,j}(s) \, ds \right) .
\end{align}
The defining equations for $\ScaleZ_{C,j}$ and $\ScaleZ_{D,j}$ are
similar. Thus the limit in this case is not a trivial process,
stressing that non-trivial models can be described within the
framework.

\begin{figure}
  \includegraphics{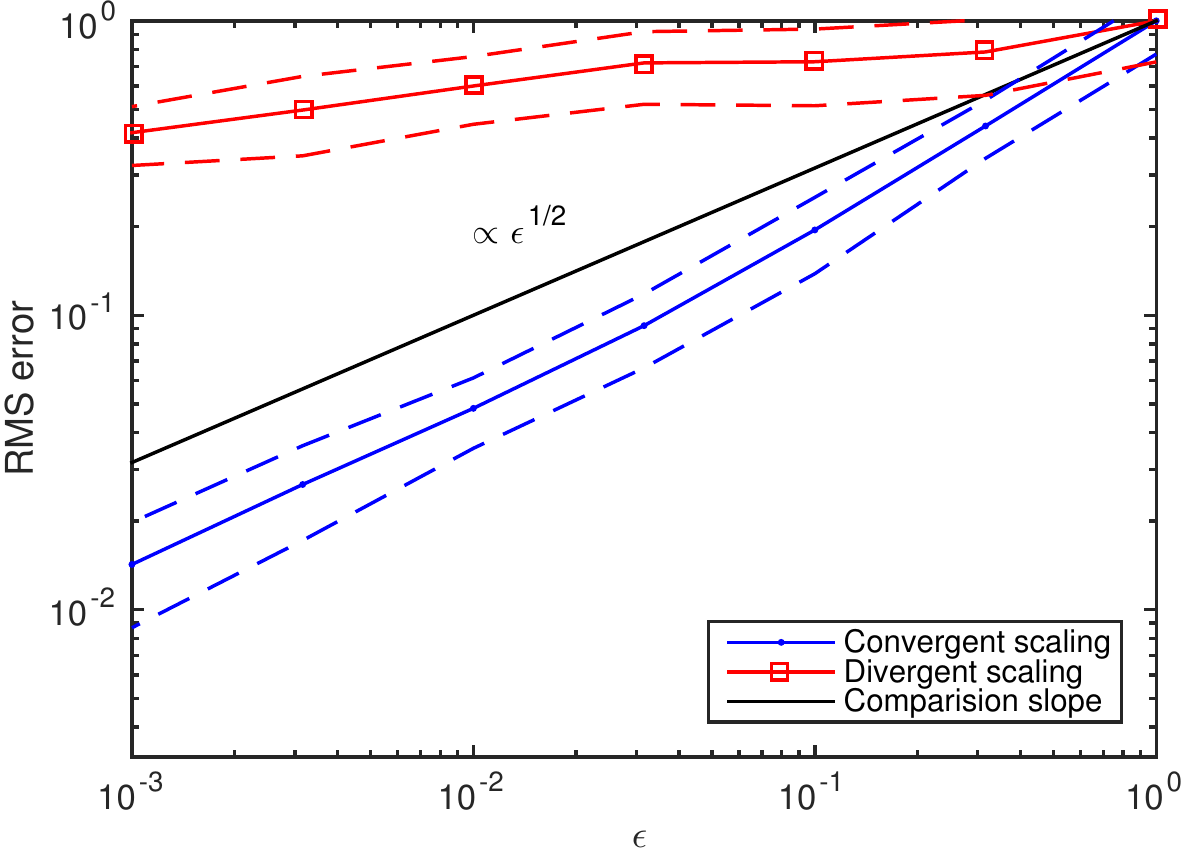}
  \caption{Multiscale errors (catalytic reactions).}
  \label{fig:cat1}
\end{figure}

\begin{figure}
  \includegraphics{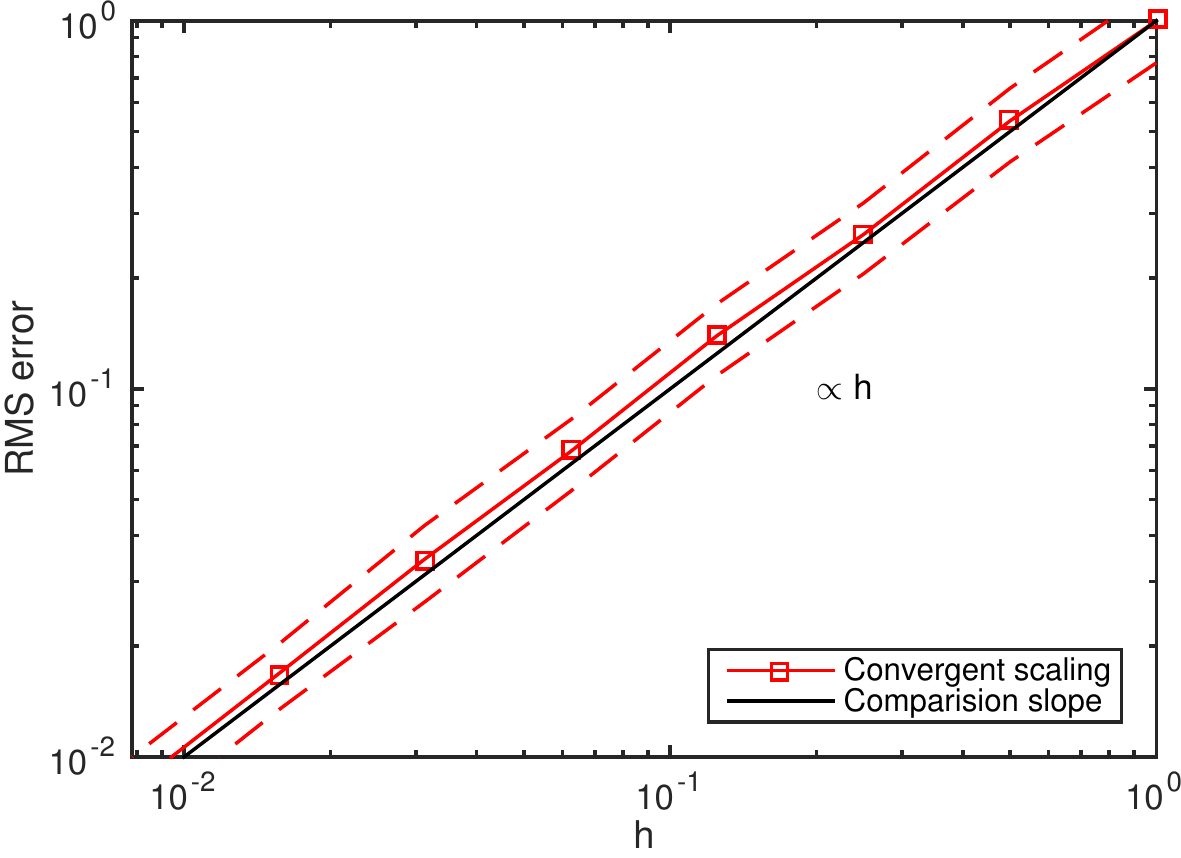}
  \caption{Split-step error (catalytic reactions). Case of
    superconvergence of the split-step method.}
  \label{fig:cat2}
\end{figure}

\subsection{Catalytic reactions: case of unclear scale separation}

It is interesting to turn the scales of the catalytic model around. If
we instead let species $A$ and $C$ be $O(1)$, while $B$ and $D$ are
$O(\epsilon^{-1})$, the topology does not change and we still have a
closed system. We put $k = 0.01 \epsilon^{1/4}$, $(k_b,k_d) =
(1,0.9)$, and use the slow diffusion $\sigma_{A,C} = \sigma_{B,D} =
\epsilon$. The critical exponents become $u = -3/4$ and $v = 0$ and
thus none of the results apply. Although Figure~\ref{fig:cat1}
(\emph{``divergent scaling''}) does not strictly exclude the
possibility of convergence, the error certainly does not go down
convincingly.



\section{Conclusions}
\label{sec:conclusions}

In this paper we have developed a coherent framework for analyzing
certain multiscale methods for continuous-time Markov chains of a
general spatial structure. Concrete assumptions and conditions have
been discovered that enables a multiscale description and a consistent
formulation of the approximating methods in operational time. Notably,
through explicit \textit{a priori} results, all processes are
well-posed and the framework does not rely on any heuristic prior
bounds.

The analysis distinguishes between two separate sources of errors,
namely \emph{the multiscale error} and \emph{the split-step
  error}. The first is due to an approximate stochastic/deterministic
variable splitting strategy, a kind of stochastic homogenization
technique. The second emerges when this approximating process in turn
is evolved in discrete time-steps. Notably, we found theoretically how
the split-step error is composed of factors remindful of the terms
making up the multiscale error, thus connecting the two in a
qualitative sense. The behavior of these errors were also examined
experimentally via actual implementations of the methods. Although
some of the boundary cases are difficult to handle theoretically, in
particular when confronted with open systems, the numerical
experiments support the sharpness of our theoretical predictions.

The work opens up for some interesting possibilities. Clearly, an
ideal implementation should allow the split-step error to be about as
large as the multiscale error. The fully discrete approximation is
amenable to several efficient algorithms developed for numerical
methods for partial differential equations, including for example
multigrid techniques. An interesting challenge to which we would like
to return is to develop practical procedures for computing accurate
error estimates. We believe this is doable following the theory laid
out in the paper.


\section*{Acknowledgment}

Many detailed suggestions by the two referees have helped us to
clarify and improve the paper.

The work was supported by ENS Cachan (A.~Chevallier) and by the
Swedish Research Council within the UPMARC Linnaeus center of
Excellence (S.~Engblom).


\appendix

\section{The multiscale error}
\label{sec:mserr}

Below are the statements and proofs of the two critical lemmas used in
the proof of Theorem~\ref{th:ScalingError}. Recall the definition of
the two effective exponents $u$ and $v$ in
\eqref{eq:udef}--\eqref{eq:vdef}.

\begin{lemma}
  \label{lem:PDMPCVSquare}
  Define $\ScaleX(t)$ by \eqref{eq:RDMEPoissreprS} and $\ScaleZ(t)$ by
  \eqref{eq:RDMEPoissrepr_ex1}--\eqref{eq:RDMEPoissrepr_ex2} with
  $\ScaleX(0) = \ScaleZ(0)$ almost surely. Then under the stopping
  time $\StopT{t}$ defined in \eqref{eq:StopT},
  \begin{align*}
    \Expect[\|\ScaleZ(\StopT{t})-\ScaleX(\StopT{t})\|^2] 
    &\lec A + B \int_0^t 
    \Expect[\|\ScaleZ(\StopT{s})-\ScaleX(\StopT{s})\|^2 ] \, ds +
   C \int_0^t \Expect[\|\ScaleZ(\StopT{s})-\ScaleX(\StopT{s})\| ] \, ds,
  \end{align*}
  where expressions for $A$, $B$, and $C$ are indicated in
  \eqref{eq:bnd_coeffs} below. These bounds depend on $\epsilon$ and
  $\Sstopping$ and on the reaction topology $\stoich$,
  \begin{align}
    A &= \epsilon^{1+v}[\LipPr(\Sstopping)\Sstopping], \quad
    B = \epsilon^{2v}[\LipPr(\Sstopping)]^2, \quad
    C = \epsilon^{u}\LipPr(\Sstopping)
    [\epsilon^{u}\LipPr(\Sstopping)\Sstopping+1].
  \end{align}
\end{lemma}

To improve the readability of the proof, we use the notation ``$A \lec
B$'' to indicate that $A \le C B$ for some constant $C > 0$ which is
$O(1)$ with respect to $\epsilon$, $\Sstopping$, and $h$. When the
processes are assumed to be bounded \textit{a priori}, clearly,
$\LipPr(c\Sstopping) \lec 1$, for any constant $c > 0$. In the
unbounded case, Assumption~\ref{ass:lip} yields similarly
$\LipPr(c\Sstopping) \lec \LipPr(\Sstopping)$ for any constant $c >
0$. We additionally let $(c_{\lvec},C_{\lvec})$ be the constants in
the norm equivalence
\begin{align}
  c_{\lvec} \pnorm{X} \le \lnorm{X} \le C_{\lvec} \pnorm{X}.
\end{align}

\begin{proof}
  We focus first on a single voxel $j$ and analyze the errors on
  species from $\Group_1(j)$ and $\Group_2(j)$, respectively. For $i
  \in \Group_1(j)$, from \eqref{eq:RDMEPoissreprS} and
  \eqref{eq:RDMEPoissrepr_ex1},
\begin{align*}
  \left( \ScaleZ_{ij}(\StopT{t}) - \ScaleX_{ij}(\StopT{t})\right)^2 = 
  \Bigg[ 
  -&\sum_{r = 1}^{R} \stoich_{ri} \left( \Pi_{rj}  
  \left( \cdot \right)-\Pi_{rj} \left( \cdot \right) \right) \\
  \nonumber
  -&\sum_{k = 1}^{J} \left( \Pi_{ijk}' 
  \left( \cdot \right)
  - \Pi_{ijk}' \left( \cdot \right) \right)+
  \sum_{k = 1}^{J} \left( \Pi_{ikj}'
  \left( \cdot \right) 
  -\Pi_{ikj}'
  \left( \cdot \right) \right) \Bigg]^2
\end{align*}
where we have suppressed the local time arguments of the Poisson
processes, available in \eqref{eq:RDMEPoissreprS} and
\eqref{eq:RDMEPoissrepr_ex1}.

By Jensen's inequality and the bound on the mesh connectivity in
Definition~\ref{def:mesh} \eqref{eq:connectivity} we get
\begin{align*}
  \left( \ScaleZ_{ij}(\StopT{t}) - \ScaleX_{ij}(\StopT{t})\right)^2 \leq
  (R + 2M_D)\left( A_1 + A_2 + A_3 \right),
\end{align*}
where in terms of
\begin{align*}
  A_1 &= \sum_{r = 1}^{R} \stoich_{ri}^2 \Biggl( \Pi_{rj}  
  \left( \int_{0}^{\StopT{t}} 
  \epsilon^{-\nu_r} V_j \bar{u}_r(V_j^{-1}\ScaleZ_{\cdot,j}(s)) \, ds \right) \\
  &\hphantom{= \sum_{r = 1}^{R} \stoich_{ri}^2 } 
  -\Pi_{rj} \left( \int_{0}^{\StopT{t}} \epsilon^{-\nu_r} V_j \bar{u}_r(V_j^{-1} 
  \ScaleX_{\cdot,j}(s)) \, ds \right) \Biggr)^2, \\
  A_2 &= \sum_{k = 1}^{J} \left( \Pi_{ijk}' 
  \left( \int_{0}^{\StopT{t}}
  \epsilon^{-\mu_i} \bar{q}_{ijk}  \ScaleZ_{ij}(s)  \, ds \right)
  - \Pi_{ijk}' \left( \int_{0}^{\StopT{t}}
  \epsilon^{-\mu_i} \bar{q}_{ijk}  \ScaleX_{ij}(s)  \, ds \right) \right)^2, \\
  A_3 &=\sum_{k = 1}^{J} \left( \Pi_{ikj}'
  \left( \int_{0}^{\StopT{t}}
  \epsilon^{-\mu_i} \bar{q}_{ikj} \ScaleZ_{ik}(s) \, ds \right) 
  -\Pi_{ikj}'
  \left( \int_{0}^{\StopT{t}}
  \epsilon^{-\mu_i} \bar{q}_{ikj} \ScaleX_{ik}(s) \, ds \right)
        \right)^2.
\end{align*}

First we need to bound the $\lvec$-norm:
\begin{align*}
  \lnorm{V_j^{-1}\ScaleZ_{\cdot,j}(s)} &\le 
  C_{\lvec} \pnorm{V_j^{-1}\ScaleZ_{\cdot,j}(s)} \le  
  C_{\lvec} V_j^{-1} \Sstopping \le 
  C_{\lvec} m_V^{-1} \bar{V}_M^{-1} \Sstopping.
\end{align*}

Then using the Lipschitz bound \eqref{eq:lipS} in
Assumption~\ref{ass:rboundS}:
\begin{align*}
  \bar{u}_r(V_j^{-1}\ScaleZ_{\cdot,j}(s)) &\le \bar{u}_r(0) + \LipPr_r( C_{\lvec} m_V^{-1} \bar{V}_M^{-1} \Sstopping) \lnorm{V_j^{-1}\ScaleZ_{\cdot,j}(s)} \\
  &\le \bar{u}_r(0) + \LipPr_r( C_{\lvec} m_V^{-1} \bar{V}_M^{-1} \Sstopping) C_{\lvec} \pnorm{V_j^{-1}\ScaleZ_{\cdot,j}(s)} \\
  &\le \bar{u}_r(0) + \LipPr_r( C_{\lvec} m_V^{-1} \bar{V}_M^{-1} \Sstopping) C_{\lvec} V_j^{-1}\Sstopping \\
  &\lec 1 + \LipPr_r(\Sstopping) V_j^{-1}\Sstopping.
\end{align*}
Thus, 
\begin{align*}
  \int_{0}^{\StopT{t}} \epsilon^{-\nu_r} V_j \bar{u}_r(V_j^{-1}\ScaleZ_{\cdot,j}(s)) \, ds 
  &\lec  \int_{0}^{\StopT{t}} \epsilon^{-\nu_r} \left( V_j + \LipPr_r(\Sstopping)\Sstopping \right) \, ds \\
  &\lec  t \epsilon^{-\nu_r} ( M_V \bar{V}_M + \LipPr_r(\Sstopping)\Sstopping) \, ds \\
  &\lec  \epsilon^{-\nu_r} (1 + \LipPr_r(\Sstopping)\Sstopping) \, ds.
\end{align*}
Using the same method for $\ScaleX$, we conclude
\begin{align*}
  \int_{0}^{\StopT{t}}& \epsilon^{-\nu_r} V_j
  \bar{u}_r(V_j^{-1}\ScaleZ_{\cdot,j}(s)) \, ds \, \vee \, 
  \int_{0}^{\StopT{t}} \epsilon^{-\nu_r} V_j
  \bar{u}_r(V_j^{-1}\ScaleX_{\cdot,j}(s)) \, ds \\
  &\lec \epsilon^{-\nu_{r}} (1+\LipPr_r(\Sstopping)  \Sstopping).
\end{align*}
Hence using Lemma~\ref{lem:Stopping} \eqref{eq:poisson_rmk} and again
the Lipschitz bound we get
\begin{align*}
  \Expect[A_1] 
  &\lec \sum_r \stoich_{ri}^2 \epsilon^{-\nu_{r}}  \LipPr_r(\Sstopping)
  \left(\epsilon^{-\nu_{r}}(\LipPr_r(\Sstopping) \Sstopping+1)+1 \right)
  \int_0^t \Expect \left[ \|\ScaleZ(\StopT{s})-\ScaleX(\StopT{s})\| \right]
  \, ds.
\end{align*}
Relying on the same arguments we readily find
\begin{align*}
  \Expect[A_2] &\lec \sum_{k = 1}^{J} \epsilon^{-\mu_i}
  \left( 1 + \epsilon^{-\mu_i} \right)
  \int_0^t \Expect[ \|\ScaleZ(\StopT{s})-\ScaleX(\StopT{s})\|] \, ds,
\end{align*}
and the identical bound for $\Expect[A_3]$.

For $i\in \Group_2(j)$, we similarly get 
\begin{align*}
  \left( \ScaleZ_{ij}(\StopT{t}) - \ScaleX_{ij}(\StopT{t}) \right)^2 \leq
  \epsilon^{2} (R + 2M_D) (A'_1 + A'_2 + A'_3)
\end{align*}
where
{\small \begin{align*}
  A'_1 &=  \sum_{r = 1}^{R} \stoich_{ri}^2 \left( \int_{0}^{\StopT{t}}  \epsilon^{-\nu_r} V_j \bar{u}_r(V_j^{-1}\ScaleZ_{\cdot,j}(s)) \, ds 
  -\Pi_{rj} \left( \int_{0}^{\StopT{t}} \epsilon^{-\nu_r} V_j \bar{u}_r(V_j^{-1}\ScaleX_{\cdot,j}(s)) \, ds\right) \right)^2, \\
  A'_2 &= \sum_{k = 1}^{J} \left( \int_{0}^{\StopT{t}} 
  \epsilon^{ -\mu_i} \bar{q}_{ijk} \ScaleZ_{ij}(s)  \, ds 
  -\Pi_{ijk}' \left(  \int_{0}^{\StopT{t}}  
  \epsilon^{-\mu_i} \bar{q}_{ijk} \ScaleX_{ij}(s)  \, ds 
  \right) \right)^2, \\
  A'_3 &=\sum_{k = 1}^{J} \left( \int_{0}^{\StopT{t}} 
  \epsilon^{ -\mu_i} \bar{q}_{ikj} \ScaleZ_{ik}(s) \, ds  
  -\Pi_{ikj}' \left( \int_{0}^{\StopT{t}} 
  \epsilon^{ -\mu_i} \bar{q}_{ikj} \ScaleX_{ik}(s) \, ds 
  \right)  \right)^2.
\end{align*}}
The analysis is now slightly different. Species from the second group
have a large number of molecules, so $\ScaleX_{ij}(t)$ is expected to
remain close to its mean value. We thus introduce the centered Poisson
processes $\tilde{\Pi}_r$,
\begin{align*}
  A'_1 &= \sum_{r} \stoich_{ri}^2 \Bigg( \int_{0}^{\StopT{t}}  \epsilon^{-\nu_r} V_j
         \left( \bar{u}_r(V_j^{-1}\ScaleZ_{\cdot,j}(s))-
         \bar{u}_r(V_j^{-1}\ScaleX_{\cdot,j}(s)) \right) \, ds \\
  &\hphantom{\sum_{r} \stoich_{ri}^2}-\tilde{\Pi}_{rj}  
  \left( \int_{0}^{\StopT{t}}  \epsilon^{-\nu_r}
  V_j \bar{u}_r(V_j^{-1}\ScaleX_{\cdot,j}(s)) \, ds \right) \Bigg)^2 \\
       &\leq \sum_{r} 2\stoich_{ri}^2 \left(
	  \int_{0}^{\StopT{t}}  \epsilon^{-\nu_r} V_j
	  \left(
         \bar{u}_r(V_j^{-1}\ScaleZ_{\cdot,j}(s))-\bar{u}_r(V_j^{-1}\ScaleX_{\cdot,j}(s))
         \right) \, ds \right)^2 \\
  &\hphantom{\sum_{r}}+2\stoich_{ri}^2 \left( \tilde{\Pi}_{rj}  
         \left( \int_{0}^{\StopT{t}}  \epsilon^{-\nu_r}
         V_j \bar{u}_r(V_j^{-1}\ScaleX_{\cdot,j}(s)) \, ds \right) \right)^2. \\
\end{align*}
Using that the quadratic variation of $\tilde{\Pi}$ is
$[\tilde{\Pi}]_t = \Pi(t)$ and the martingale stopping time theorem we
get
\begin{align*}
  \Expect &\left[\left( \tilde{\Pi}_{rj} 
            \left( \int_{0}^{\StopT{t}}  \epsilon^{-\nu_r} V_j 
            \bar{u}_r(V_j^{-1}\ScaleX_{\cdot,j}(s)) \, ds \right) \right)^2  \right]
  \\
  &= \Expect\left[ \Pi_{rj} \left( \int_{0}^{\StopT{t}}
  \epsilon^{-\nu_r} V_j \bar{u}_r(V_j^{-1}\ScaleX_{\cdot,j}(s)) \, ds
  \right) \right] \\
  &= \Expect\left[ \int_{0}^{\StopT{t}}
  \epsilon^{-\nu_r} V_j \bar{u}_r(V_j^{-1}\ScaleX_{\cdot,j}(s)) \, ds
  \right] \lec \epsilon^{-\nu_{r}} \LipPr_r(\Sstopping) \Sstopping.
\end{align*}
\noindent
Using Cauchy-Schwartz for the remaining integral part and following
the same approach for $A'_2$ and $A'_3$ we get
\begin{align*}
  \Expect[A'_1] &\lec  \sum_{r = 1}^R \stoich^2_{ri} \epsilon^{-\nu_{r}} \LipPr_r(\Sstopping) \Sstopping+
   \sum_{r = 1}^R \stoich^2_{ri} \left( \epsilon^{-\nu_{r}} \LipPr_r(\Sstopping)  \right)^2 \int_0^{t} \Expect[\|\ScaleZ(\StopT{s})-\ScaleX(\StopT{s})\|^2] \, ds, \\
  \Expect[A'_2] &\lec  \sum_{k = 1}^J \epsilon^{-\mu_i} + 
   \sum_{k = 1}^J \left(\epsilon^{-\mu_i} \right)^2 \int_0^{t} \Expect[\|\ScaleZ(\StopT{s})-\ScaleX(\StopT{s})\|^2] \, ds,
\end{align*}
as well as an identical bound for $\Expect[A'_3]$.

We thus get for the $j$th voxel,
\begin{align}
  \nonumber
  &\Expect \left[ \|\ScaleZ_{\cdot,j}(t) - \ScaleX_{\cdot,j}(t)\|^2\right] \lec  
  \sum_{i \in \Group_1(j)} \Expect\left[A_1 + A_2 + A_3 \right]+
  \sum_{i \in \Group_2(j)} \epsilon^{2}
  \Expect\left[A'_1 + A'_2 + A'_3 \right] \\
  \label{eq:bnd_coeffs}
  &\hphantom{\Expect }
    \lec A^{(j)} + B^{(j)} \int_0^{t}
      \Expect[\|\ScaleX(\StopT{s})-\ScaleZ(\StopT{s})\|^2] \, ds
      +C^{(j)} \int_0^{t} \Expect[\|\ScaleX(\StopT{s})-\ScaleZ(\StopT{s})\|] \, ds.
\end{align}
Summing over $j$ we get the stated result with
$A := \sum_{j} A^{(j)}$, $B := \sum_{j} B^{(j)}$, and
$C := \sum_{j} C^{(j)}$.
\end{proof}

\begin{lemma}
  \label{lem:PDMPCV}
  Under the same assumptions as in Lemma~\ref{lem:PDMPCVSquare},
  \begin{align*}
    \Expect[\|\ScaleX(\StopT{t})-\ScaleZ(\StopT{t})\|]
    &\lec D + E \int_{0}^{t}
    \Expect[\|\ScaleX(\StopT{s}) - \ScaleZ(\StopT{s})\| ] \, ds,
  \end{align*}
  where explicit expressions for $D$ and $E$ are found in
  \eqref{eq:bnd_coeffs2} below and depend on $\epsilon$, $\Sstopping$,
  and on the reaction topology $\stoich$,
  \begin{align}
    D &= \epsilon^{1/2+v/2}[\LipPr(\Sstopping)\Sstopping]^{1/2}, \quad
    E = [\epsilon^{u}+\epsilon^{v}]\LipPr(\Sstopping).
  \end{align}
\end{lemma}

\begin{proof}
For voxel $j$ and for $i \in \Group_1(j)$,
\begin{align*}
  \left| \ScaleZ_{ij}(\StopT{t}) - \ScaleX_{ij}(\StopT{t}) \right| \leq  
  &\sum_{r = 1}^{R} \left| \stoich_{ri} \left( \Pi_{rj}  
  \left( \cdot \right)-\Pi_{rj} \left( \cdot \right) \right) \right| \\
  +&\sum_{k = 1}^{J} \left| \left( \Pi_{ijk}' 
  \left( \cdot \right)
  - \Pi_{ijk}' \left( \cdot \right) \right) \right|+
  \sum_{k = 1}^{J} \left| \left( \Pi_{ikj}'
  \left( \cdot \right) 
  -\Pi_{ikj}'
  \left( \cdot \right) \right) \right|.
\end{align*}
\noindent
We keep the same notation as in the previous lemma and thus write
\begin{align*}
  &\left| \ScaleZ_{ij}(\StopT{t}) - \ScaleX_{ij}(\StopT{t}) \right| \leq 
  \left( A_1 + A_2 + A_3 \right), \\
  \intertext{where}
  \Expect[A_1 ] &= \sum_{r} |\stoich_{ri}| \Expect \left[ \left| \Pi_{rj}  
    \left( \cdot \right)-\Pi_{rj} \left( \cdot \right) \right | \right] \\
  &= \sum_{r} |\stoich_{ri}| \Expect \left[ \left| 
    \int_{0}^{\StopT{t}} \epsilon^{-\nu_r} \VoxVol_j \bar{u}_r(V_j^{-1}\ScaleZ_{\cdot,j}(s)) \, ds 
    - \int_{0}^{\StopT{t}} \epsilon^{-\nu_r} \VoxVol_j \bar{u}_r(V_j^{-1}\ScaleX_{\cdot,j}(s)) \, ds 
    \right | \right]\\
  &\lec \sum_{r} |\stoich_{ri}| \epsilon^{-\nu_{r}} \LipPr_r(\Sstopping)
  \int_{0}^{t}\Expect\left[\|\ScaleZ_{\cdot,j}(\StopT{s})-\ScaleX_{\cdot,j}(\StopT{s})\|\right].
\end{align*}
\noindent
In the same spirit we find
\begin{align*}
  \Expect[A_2] &\lec \sum_{k = 1}^{J} \epsilon^{-\mu_i}
  \int_{0}^{t} \Expect\left[\|\ScaleZ_{\cdot,j}(\StopT{s})-\ScaleZ_{\cdot,j}(\StopT{s})\|\right],
\end{align*}
and an identical bound for $\Expect[A_3]$.

Continuing with $i \in \Group_2(j)$,
\begin{align*}
  \left( \ScaleZ_{ij}(\StopT{t}) - \ScaleX_{ij}(\StopT{t}) \right) \leq
  \epsilon (A'_1 + A'_2 + A'_3),
\end{align*}
where
\begin{align*}
  A'_1 &= \sum_{r = 1}^{R} |\stoich_{ri}| \left| \int_{0}^{\StopT{t}} \VoxVol_j \epsilon^{-\nu_r} \bar{u}_r(V_j^{-1}\ScaleZ_{\cdot,j}(s)) \, ds 
  -\Pi_{rj} \left( \int_{0}^{\StopT{t}} \VoxVol_j \epsilon^{-\nu_r} \bar{u}_r(V_j^{-1}\ScaleX_{\cdot,j}(s)) \, ds\right) \right|, \\
  A'_2 &= \sum_{k = 1}^{J} \left| \int_{0}^{\StopT{t}} 
  \epsilon^{-\mu_i}\bar{q}_{ijk} \ScaleZ_{ij}(s)  \, ds 
  -\Pi_{ijk}' \left(  \int_{0}^{\StopT{t}}
  \epsilon^{-\mu_i} \bar{q}_{ijk} \ScaleX_{ij}(s)  \, ds 
  \right) \right|, \\
  A'_3 &=\sum_{k = 1}^{J} \left| \int_{0}^{\StopT{t}}
 \epsilon^{-\mu_i} \bar{q}_{ikj} \ScaleZ_{ik}(s) \, ds  
  -\Pi_{ikj}' \left( \int_{0}^{\StopT{t}}
  \epsilon^{-\mu_i}\bar{q}_{ikj} \ScaleX_{ik}(s) \, ds 
  \right)  \right|.
\end{align*}

Using the same techniques developed previously we find
{\small \begin{align*}
  \Expect[A'_1] &\leq \Expect \Biggl[ \sum_{r} |\stoich_{ri}| 
    \int_{0}^{\StopT{t}} \epsilon^{-\nu_{r}} C_{\lvec} \LipPr_r(V_j^{-1}\Sstopping) \|\ScaleZ(s)-\ScaleX(s)\| \, ds \\
    &\hphantom{\leq \qquad}
    +|\stoich_{ri}|\left| \tilde{\Pi}_{rj} \left( \int_{0}^{\StopT{t}} \VoxVol_j \epsilon^{-\nu_r} \bar{u}_r(V_j^{-1}\ScaleX_{\cdot,j}(s)) \, ds\right) 
    \right| \Biggr]\\
  &\lec \sum_{r} |\stoich_{ri}| 
  \epsilon^{-\nu_{r}} \LipPr_r(\Sstopping) \int_{0}^{t} \Expect[\|\ScaleZ(\StopT{s})-\ScaleX(\StopT{s})\| ]\, ds \\
  &\hphantom{\leq \qquad}
  +|\stoich_{ri}|\Expect\left[ \tilde{\Pi}^2_{rj} 
    \left( \int_{0}^{t} \VoxVol_j \epsilon^{-\nu_r} \bar{u}_r(V_j^{-1}\ScaleX_{\cdot,j}(\StopT{s})) \, ds\right) \right]^{1/2} \\
  &\lec \sum_{r} |\stoich_{ri}| 
  \epsilon^{-\nu_{r}} \LipPr_r(\Sstopping) \int_{0}^{t} \Expect[\|\ScaleZ(\StopT{s})-\ScaleX(\StopT{s})\| ]\, ds \\
  &\hphantom{\leq \qquad}
  +|\stoich_{ri}|\Expect\left[  
    \left( \int_{0}^{t} \VoxVol_j \epsilon^{-\nu_r} \bar{u}_r(V_j^{-1}\ScaleX_{\cdot,j}(\StopT{s})) \, ds\right) \right]^{1/2}\\
  &\lec \sum_{r} |\stoich_{ri}| 
  \epsilon^{-\nu_{r}} \LipPr_r(\Sstopping) \int_{0}^{t} \Expect[\|\ScaleZ(\StopT{s})-\ScaleX(\StopT{s})\| ]\, ds 
  + |\stoich_{ri}| \left( t \epsilon^{-\nu_{r}} C_{\lvec} \LipPr_r(\Sstopping) \Sstopping \right)^{1/2} \\
  &\lec \sum_{r} |\stoich_{ri}| \left(
  \epsilon^{-\nu_{r}} \LipPr_r(\Sstopping) \int_{0}^{t} \Expect[\|\ScaleZ(\StopT{s})-\ScaleX(\StopT{s})\| ]\, ds 
  + \left(\epsilon^{-\nu_{r}} \LipPr_r(\Sstopping) \Sstopping \right)^{1/2} \right).
\end{align*}}

In much the same spirit we get
\begin{align*}
  \Expect[A'_2] &\lec \sum_{k = 1}^J \epsilon^{-\mu_i} \int_{0}^{t} \Expect[\|\ScaleZ(\StopT{s})-\ScaleX(\StopT{s})\| ] \, ds+ \left( \epsilon^{-\mu_i}\right)^{1/2},
\end{align*}
along with an identical bound for $\Expect[A'_3]$.

Combined we thus get for the $j$th voxel,
\begin{align}
  \nonumber
  \Expect \left[ \|\ScaleZ_{\cdot,j}(\StopT{t}) - \ScaleX_{\cdot,j}(\StopT{t})\|\right] &\lec  
  \sum_{i \in \Group_1(j)} \Expect\left[A_1 + A_2 + A_3 \right]+
  \sum_{i \in \Group_2(j)} \epsilon
  \Expect\left[A'_1 + A'_2 + A'_3 \right] \\
  \label{eq:bnd_coeffs2}
  &\hphantom{\Expect }
    \lec D_j + E_j \int_0^{t}
      \Expect[\|\ScaleX(\StopT{s})-\ScaleZ(\StopT{s})\|] \, ds.
\end{align}
Summing over $j$ we get the stated result.
\end{proof}


\section{The split-step error}
\label{sec:sserr}

The consistency of the numerical split-step method hinges on the
regularity of the kernel function $\sigma_{h}(s)$. The following lemma
(borrowed from \cite[Lemma~3.7]{jsdesplit}) paired with the strong
regularity of the involved processes provides for the order estimate
in Theorems~\ref{th:spliterr2} and \ref{th:spliterr1}. Note that the
result can be thought of as \cadlag-version of the Riemann-Lebesgue
lemma.

\begin{lemma}{(\cite[Lemma~3.7]{jsdesplit})}
  \label{lem:sigma2}
  Let $G:\Realdom^{D} \to \Realdom$ be a globally Lipschitz continuous
  function with Lipschitz constant $L$ and let $f:\Realdom \to
  \Realdom^{D}$ be a piecewise constant \cadlag\ function. Then
  \begin{align}
    \label{eq:sigma_lemma2}
    \left| \int_{0}^{t} \sigma_{h}(s) G(f(s)) \, ds \right| &\le
    \frac{h}{2} |G(f(t))|+
    \frac{h}{2} L V_{[0,t]}(f),
  \end{align}
  where the total absolute variation may be exchanged with the square
  root of the quadratic variation $[f]_{t}^{1/2}$. If $t$ is a
  multiple of $h$, then the first term on the right side of
  \eqref{eq:sigma_lemma2} vanishes.
\end{lemma}

The proofs of the following two lemmas follow closely the ideas in the
proofs of Lemmas~\ref{lem:PDMPCVSquare} and \ref{lem:PDMPCV}, but
using in addition Lemma~\ref{lem:sigma2} and Theorem~\ref{th:exist3}
to bound certain additional terms.

\begin{lemma}
  \label{lem:SplitStepCVSquare}
  Define $\ScaleZ(t)$ by
  \eqref{eq:RDMEPoissrepr_ex1}--\eqref{eq:RDMEPoissrepr_ex2} and
  $\ScaleY(t)$ by
  \eqref{eq:RDMEPoissrepr_num1}--\eqref{eq:RDMEPoissrepr_num2} with
  $\ScaleZ(0) = \ScaleY(0)$ almost surely. The under the stopping time
  $\StopT{t}$ defined in \eqref{eq:StopT}, for a fixed $\epsilon$ and
  $h$ small enough,
  \begin{align*}
    \Expect[\|\ScaleZ(\StopT{t})-\ScaleY(\StopT{t})\|^2]
    &\lec A + B \int_{0}^{t}
    \Expect[\|\ScaleZ(\StopT{s}) - \ScaleY(\StopT{s})\|^2 ] \, ds \\
    &\hphantom{\lec A}+ C \int_{0}^{t}
    \Expect[\|\ScaleZ(\StopT{s}) - \ScaleY(\StopT{s})\| ] \, ds,
  \end{align*}
  where 
  \begin{align*}
    A &= \epsilon^u \LipPr(\Sstopping)
    [\epsilon^u\LipPr(\Sstopping)\Sstopping+1] h+
    \epsilon^{2v} [\LipPr(\Sstopping)]^2 h^2, \\
    B &= \epsilon^{2v} [\LipPr(\Sstopping)]^2, \quad
    C = \epsilon^u\LipPr(\Sstopping)
    [\epsilon^u\LipPr(\Sstopping)\Sstopping+1].
  \end{align*}
\end{lemma}

\begin{lemma}
  \label{lem:SplitStepCV}
  Under the same assumptions as in Lemma~\ref{lem:SplitStepCVSquare},
  \begin{align*}
    \Expect[\|\ScaleZ(\StopT{t})-\ScaleY(\StopT{t})\|]
    &\lec D + E \int_{0}^{t}
    \Expect[\|\ScaleZ(\StopT{s}) - \ScaleY(\StopT{s})\| ] \, ds,
  \end{align*}
  with
  \begin{align*}
    D &= [\epsilon^u + \epsilon^v]\LipPr(\Sstopping) h, \quad
    E =  [\epsilon^u + \epsilon^v]\LipPr(\Sstopping).
  \end{align*}
\end{lemma}


\newcommand{\doi}[1]{\href{http://dx.doi.org/#1}{doi:#1}}
\newcommand{\available}[1]{Available at \url{#1}}
\newcommand{\availablet}[2]{Available at \href{#1}{#2}}

\bibliographystyle{abbrvnat}
\bibliography{references}

\end{document}